\documentclass[a4paper]{article}
\usepackage{amssymb, latexsym, amsmath, amsthm}
\usepackage{a4wide}

\newtheorem{theorem}{Theorem}[section]
\newtheorem{lemma}[theorem]{Lemma}
\newtheorem{proposition}[theorem]{Proposition}

\newtheorem{corollary}[theorem]{Corollary}

\theoremstyle{definition}

\newtheorem{remark}[theorem]{Remark}

\newcommand{\K}{\mathbb{K}}

\newcommand{\Z}{\mathbb{Z}}
\newcommand{\N}{\mathbb{N}}
\newcommand{\Q}{\mathbb{Q}}
\newcommand{\h}{\mathcal{H}}

\newcommand{\q}{\quad}

\newcommand{\germ}{\mathfrak}

\newcommand{\aut}{{\rm Aut}_{\K}(L)}
\newcommand{\aaut}{{\rm Aut}_{\K}(A)}

\begin{document}

\title{Automorphisms of Generalized Down-Up Algebras}
\author{Paula A.A.B. Carvalho${}^{*}$ and Samuel A. Lopes\thanks{Work partially supported  by \emph{Centro de Matem\'atica da Universidade do Porto} (CMUP),  financed by FCT (Portugal) through the programmes POCTI (\emph {Programa Operacional Ci\^encia, Tecnologia, Inova\c c\~ao}) and POSI  (\emph{Programa Operacional Sociedade da Informa\c c\~ao}), with  national and European Community structural funds.}}
\date{}
\maketitle

\begin{abstract}
A generalization of down-up algebras was introduced by Cassidy and 
Shelton in~\cite{CS04}, the so-called generalized down-up 
algebras.
We describe the automorphism group of conformal Noetherian generalized down-up algebras 
$L(f,r,s,\gamma)$ such that $r$ is not a root of unity, listing explicitly the 
elements of the group. In the last section we apply these results to Noetherian 
down-up algebras, thus obtaining a characterization of the automorphism group of 
Noetherian down-up algebras $A(\alpha, \beta, \gamma)$ for which the roots of 
the polynomial  $X^2-\alpha  X-\beta$ are not both  roots of 
unity.
\end{abstract}

\\$ $

\section*{Introduction}\label{S:int}

Generalized down-up algebras were introduced by Cassidy and Shelton in~\cite{CS04} as a generalization of the down-up algebras $A(\alpha, \beta, \gamma)$ of Benkart and Roby~\cite{BR98}. Generalized down-up algebras include all down-up algebras, the algebras \emph{similar to the enveloping algebra of $\mathfrak{sl}_{2}$} defined by Smith~\cite{spS90}, Le Bruyn's \emph{conformal $\mathfrak{sl}_{2}$ enveloping algebras}~\cite{lLB95} and Rueda's algebras \emph{similar to the enveloping algebra of $\mathfrak{sl}_{2}$}~\cite{sR02}. The reader is encouraged to consult~\cite{CS04} for further details and references.

Two of the most remarkable examples of down-up algebras are $U(\mathfrak{sl}_{2})$ and 
$U(\mathfrak{h})$, the enveloping algebras of the $3$-dimensional complex simple Lie algebra 
$\mathfrak{sl}_{2}$ and of the $3$-dimensional nilpotent, non-abelian Heisenberg Lie algebra 
$\mathfrak{h}$, respectively. These algebras have a very rich structure and representation theory which has been extensively studied, having an unquestionable impact on the theory of semisimple and nilpotent Lie algebras. Nevertheless, a precise description of their symmetries, as given by the understanding of their automorphism group, is yet to be obtained (see~\cite{jD68, jD73} and~\cite{aJ76, jA86}). The problem of describing the automorphism group seems to be considerably simpler when a deformation is introduced. Indeed, the automorphism group of the quantized enveloping algebra 
$U_{q}(\mathfrak{sl}_{2})$ was computed in~\cite{AC92}, and in~\cite{pC95, AD96} the authors independently described the group of automorphisms of the quantum Heisenberg algebra; in all cases it was assumed that the deformation parameter is not a root of unity. Despite these and  other successful results on the description of automorphism groups of  quantum algebras, e.g. \cite{AC92, AD96, pC95, GTEK02, sLn07, LnLe05pr, LnL07pr}, there is yet much to be done. For example, regarding the quantized enveloping algebras $U_{q}(\mathfrak{g}^{+})$, where $\mathfrak{g}$ is a finite-dimensional complex simple Lie algebra and $\mathfrak{g}^{+}$ is a maximal nilpotent subalgebra of 
$\mathfrak{g}$, there is a conjecture of Andruskiewitsch and Dumas~\cite{ADpr} describing the automorphism group of $U_{q}(\mathfrak{g}^{+})$ as a semidirect product of a torus of rank equal to the rank of $\mathfrak{g}$ by a finite group corresponding to the automorphisms of the Dynkin diagram of 
$\mathfrak{g}$. So far, only particular cases of this conjecture have been verified, for $\mathfrak{g}$ of rank at most $3$~\cite{pC95, AD96, sLn07, LnL07pr}. Another difficulty that arises is when the deformation parameter is a root of unity. Very few results are known in this case, e.g.~\cite[Prop.\ 1.4.4, Th\'e.\ 1.4.5]{AC92} and~\cite[Sec.\ I]{AD96}.

It is reasonable to think of a Noetherian generalized down-up algebra as a deformation of an enveloping algebra of a $3$-dimensional Lie algebra. Working over an algebraically closed field of characteristic $0$, we use elementary methods to compute the automorphism groups of Noetherian generalized down-up algebras, under certain assumptions. This is the content of Theorem~\ref{T:thegp:main}. Specializing, in Section~\ref{S:adua}, our results to down-up algebras, we obtain in Theorem~\ref{T:adua:dua:main2} a complete description of the automorphism groups of all Noetherian down-up algebras $A(\alpha, \beta, \gamma)$, under the restriction that at least one of the roots of the polynomial $X^2-\alpha X-\beta$ is not a root of unity.


\section{Generalized down-up algebras}\label{S:gdua}

Throughout this paper, $\N$ is the set of nonnegative integers, $\K$ denotes an algebraically closed field of characteristic $0$ and $\K^*$ is the multiplicative group of units of $\K$. If $A$ is a subset of the ring $R$ then the two-sided ideal of $R$ generated by $A$ is denoted by  $\langle A\rangle$; we also write $\langle x_{1}, \ldots, x_{n}\rangle $ in place of 
$\langle\{ x_{1}, \ldots, x_{n}\}\rangle $.

Given a polynomial $f=a_{0}+a_{1}X+\cdots +a_{n}X^n\in\K[X]$, with all $a_{i}\in\K$, we define the support of $f$ to be the set $\mathrm{supp}\, (f)=\{ i\mid a_{i}\neq 0 \}$ and the degree of $f$, denoted $\deg (f)$, as the supremum of $\mathrm{supp}\, (f)$. In particular, the zero polynomial has degree $-\infty$, the supremum of the empty set.


\subsection{Preliminaries}\label{SS:prel}

Let $f\in\K[X]$ be a polynomial and fix scalars $r, s, \gamma\in\K$. The generalized down-up algebra $L=L(f, r, s, \gamma)$ was defined in~\cite{CS04} as the unital associative $\K$-algebra generated by $d$, $u$ and $h$, subject to the relations:
\begin{align}\label{E:defGdua1}
dh-rhd+\gamma d&=0,\\\label{E:defGdua2}
hu-ruh+\gamma u&=0,\\\label{E:defGdua3}
du-sud+f(h)&=0.
\end{align}

When $f$ has degree one,  we retrieve all down-up algebras 
$A(\alpha, \beta, \gamma)$, $\alpha$, $\beta$, $\gamma\in\K$, for suitable choices of the parameters of $L$. This is argued 
in~\cite[Ex.\ 1.2]{CS04}. To correct some typos, we will explicitly construct isomorphisms between the algebras $A(\alpha, \beta, \gamma)$ and the  algebras $L=L(f, r, s, \gamma)$, in case $f$ has degree one. The reader is referred to~\cite{BR98} for the definition of $A(\alpha, \beta, \gamma)$.

\begin{lemma}[{\cite[Ex.\ 1.2]{CS04}}]\label{L:iso:duaGdua}
\begin{enumerate}
\item Given $\alpha$, $\beta$, $\gamma\in\K$, let $r$ and $s$ be the roots of $X^2-\alpha X-\beta$. Then,
\begin{equation*}
L(X, r, s, \gamma)\simeq A(\alpha, \beta, \gamma).
\end{equation*}
\item Let $\lambda$, $\mu$, $r$, $s$, $\gamma\in\K$ with $\lambda\neq 0$. Then,
\begin{equation*}
L(\lambda X+\mu, r, s, \gamma)\simeq A(r+s, -rs, \lambda\gamma+(r-1)\mu).
\end{equation*}
\end{enumerate}
In both cases, there is an isomorphism taking the canonical generators $d$ and $u$ of $L$ to the canonical generators $d$ and $u$ of $A$, respectively. Under that isomorphism, $h$ is sent to 
$sud-du$ in case (a) and to $\lambda^{-1}(sud-du-\mu)$ in case (b).
\end{lemma}

Other natural isomorphisms between generalized down-up algebras are the following, for $\lambda\in\K^{*}$:
\begin{itemize}
\item $L(f, r, s, \gamma)\xrightarrow{\simeq}L(f(\lambda^{-1}X), r, s, \lambda\gamma)$, where  $u\mapsto u$, $d\mapsto d$, $h\mapsto \lambda^{-1}h$;
\item $L(f, r, s, \gamma)\xrightarrow{\simeq}L(\lambda f, r, s, \gamma)$, where  $u\mapsto \lambda^{-1}u$, $d\mapsto d$, $h\mapsto h$.
\end{itemize}
Therefore, if convenient, it can be assumed that either $f=0$ or $f$ is monic, and that either $\gamma=0$ or $\gamma=1$. An additional symmetry comes from an antiautomorphism of $L(f, r, s, \gamma)$ interchanging $u$ and $d$ and fixing $h$. Because of this antiautomorphism,  one can carry over properties of the generator $u$ to properties of  $d$, and vice-versa. 


\subsection{Noetherian generalized down-up algebras}\label{SS:noeth}

Several ring-theoretical and homological properties of $L$ were derived by Cassidy and Shelton~\cite[Secs.\ 2, 3]{CS04}, and in~\cite[Sec.\ 4]{CS04}  they classified all simple weight modules of $L$ under the assumption that $rs\neq 0$, which is precisely when $L$ is a Noetherian domain. This classification was later extended by Praton~\cite{iP07} to the non-Noetherian case.

Let us briefly recall some of the results from~\cite{CS04} which we will use often. We begin with 
\cite[Props.\ 2.5--2.6]{CS04}, which extend to $L$ results from~\cite{KMP99}:

\begin{proposition}\label{P:noeth:noethdomrs}
The following conditions are equivalent:
\begin{enumerate}
\item $L$ is Noetherian;
\item $L$ is a domain;
\item $rs\neq 0$.
\end{enumerate}
\end{proposition}

Given a ring $D$, an automorphism $\sigma$ of $D$ and a central element $a\in D$, the generalized Weyl algebra $D(\sigma, a)$ is the ring extension of $D$ generated by $x$ and $y$, subject to the relations:
\begin{equation}\label{E:noeth:gwa1}
xb=\sigma (b)x, \q\q by=y\sigma (b),  \q\q \text{for all $b\in D$;}
\end{equation}
\begin{equation}\label{E:noeth:gwa2}
yx=a, \q\q xy=\sigma (a).
\end{equation}
Generalized Weyl algebras were introduced and studied by Bavula~\cite{vB93}, and their properties and representation theory have been subsequently studied by himself and several other authors.
If $D$ is a Noetherian $\K$-algebra which is a domain, the automorphism $\sigma$ is $\K$-linear and $a\neq 0$ then $D(\sigma, a)$ is a Noetherian domain (see~\cite{vB93} for example).

As occurs with down-up algebras~\cite{KMP99}, the Noetherian generalized down-up algebras can be presented as generalized Weyl algebras. In fact, set $a=ud$, let $D$ be the commutative polynomial algebra 
$\K[h, a]$ and define the automorphism $\sigma$ of $D$ by the rules $\sigma(h)=rh-\gamma$ and 
$\sigma(a)=sa-f(h)$.

\begin{lemma}[{\cite[Lem.\ 2.7]{CS04}}]\label{L:noeth:gwa}
With the notation introduced above, $L$ is isomorphic to the generalized Weyl algebra $D(\sigma, a)$, under an isomorphism taking $d\in L$ (resp.\ $u$, resp.\ $h$) to $x\in D(\sigma, a)$ (resp.\ $y$, resp.\ $h$).
\end{lemma}

Let $R$ be a ring and let $\tau$ be an endomorphism of $R$. Recall that a (left) $\tau$-derivation of $R$ is an additive map $\delta : R\rightarrow R$ which satisfies the relation
$
\delta (ab)=\tau (a)\delta (b)+\delta (a) b
$
for all $a, b\in R$. Given $R$, $\tau$ and $\delta$ as above, we can form the skew polynomial ring 
$R[\theta;\tau, \delta ]$. As a left $R$-module, $R[\theta;\tau, \delta ]$ is free with basis $\{ \theta^{i} \mid i\geq 0 \}$ and the multiplication in $R[\theta;\tau, \delta ]$ is determined by that of $R$ and the rule:
$$
\theta a=\tau (a)\theta+\delta (a),
$$
for $a\in R$. Naturally, if $\tau'$ is an endomorphism of $R[\theta;\tau, \delta ]$ and $\delta'$ is a $\tau'$-derivation of $R[\theta;\tau, \delta ]$, this construction can be repeated to obtain an iterated skew polynomial ring $R[\theta;\tau, \delta ][\Phi;\tau', \delta']$, and so on.

The next remark will be useful when comparing normal elements of $L$ generating the same ideal. 

\begin{lemma}\label{L:noeth:units}
If $rs\neq 0$ then $L$ is an iterated skew polynomial ring over $\K$ and the group of units of $L$ is $\K^*$.
\end{lemma}
\begin{proof}
We can realize $L$ as the iterated skew polynomial ring
$$
\K[h][d; \sigma][u; \sigma^{-1}, \delta],
$$
where the automorphism $\sigma$ of $\K[h]$ given by $\sigma(h)=rh-\gamma$ is extended to an automorphism of $\K[h][d; \sigma]$ by defining $\sigma (d)=sd$ and the $\sigma$-derivation $\delta$ of $\K[h][d; \sigma]$ is determined by the rules $\delta (h)=0$ and $\delta (d)=s^{-1}f(h)$. Now it follows from well-known results on skew polynomial rings that the units of $L$ are just the non-zero scalars.
\end{proof}

\begin{remark}
The hypothesis $rs\neq 0$ in the previous lemma is not unnecessary. For example, if $r=0$ then the calculation
$$
(1+\gamma u+uh)(1-\gamma u-uh)=1=(1-\gamma u-uh)(1+\gamma u+uh)
$$
shows that $1+\gamma u+uh$ is a non-scalar unit of $L=L(f, 0, s, \gamma)$. 
\end{remark}


\subsection{Conformal generalized down-up algebras}\label{SS:conf}

Generalized down-up algebras can also be viewed as ambiskew polynomial rings (see~\cite[Sec.\ 2]{CS04} and~\cite{dJ00}). In this context, $L$ is said to be conformal if there exists a polynomial $g\in\K[X]$ such that $f(X)=sg(X)-g(rX-\gamma)$. One of the advantages of $L$ being conformal is that in this case the element $z=du-g(rh-\gamma)=s(ud-g(h))$  is normal and satisfies the relations $zh=hz$, $dz=szd$ and $zu=suz$; furthermore, $z$  is nonzero provided $s\neq 0$.

If $f=0$, then clearly $L$ is conformal. Otherwise, write
\begin{equation}\label{E:conf}
f(X)=a_{0}+a_{1}X+\cdots +a_{n}X^{n}
\end{equation}
with $a_{i}\in\K$, $n\geq 0$ and $a_{n}\neq 0$. Hence $\deg (f)=n$. Cassidy and Shelton~\cite[Lem.\ 2.8]{CS04} give a sufficient condition for $L$ to be conformal, namely that $s\neq r^{i}$ for all 
$0\leq i\leq n$. As is pointed out, this condition is not necessary (take for example $f(X)=X$, $r=s=\gamma=1$ and $g(X)=\frac{1}{2}(X^2+X)$). 

If $\gamma=0$ it is easy to give a necessary and sufficient condition for $L$ to be conformal. We will see shortly that, up to isomorphism, the condition $\gamma=0$ is not very restrictive.

\begin{lemma}\label{L:conf:gamma0}
Let $f$ be as in~(\ref{E:conf}). Then $L(f, r, s, 0)$ is conformal if and only if $s\neq r^{i}$ for all $i$ such that $a_{i}\neq 0$. In that case,  a polynomial $g$ satisfying  $f(X)=sg(X)-g(rX)$  exists and is unique if we impose the additional condition that  $\mathrm{supp}\, (f)=\mathrm{supp}\, (g)$; in particular, $g$ can be chosen so that $\deg(g)=\deg(f)$.
\end{lemma}
\begin{proof}
Write $g(X)=b_{0}+b_{1}X+\cdots +b_{m}X^m$, so 
$$
sg(X)-g(rX)=\sum_{i=0}^{m}(s-r^{i})b_{i}X^{i}.
$$ 
Thus, if $f(X)=sg(X)-g(rX)$ then $m\geq n$, $a_{i}=(s-r^{i})b_{i}$ for all $0\leq i\leq n$ and $(s-r^{i})b_{i}=0$ for  $i>n$. In particular, $\mathrm{supp}\, (f)\subseteq \mathrm{supp}\, (g)$ and the condition that $s\neq r^{i}$ for all $i$ such that $a_{i}\neq 0$ is necessary for $L$ to be conformal. Moreover, if this condition is satisfied and we take 
$$
g(X)=\sum_{a_{i}\neq 0}\frac{a_{i}}{s-r^{i}}\, X^{i},
$$
then we see that $L$ is indeed conformal with $f(X)=sg(X)-g(rX)$ and $\mathrm{supp}\, (f)=\mathrm{supp}\, (g)$. The uniqueness is clear from the construction.
\end{proof}

\begin{proposition}\label{P:conf:gamma0rnot1}
If $r\neq 1$ then $L(f, r, s, \gamma)\simeq L(\tilde{f}, r, s, 0)$ for some $\tilde{f}\in\K[X]$ of the same degree as $f$. Furthermore, $L(f, r, s, \gamma)$ is conformal if and only if $L(\tilde{f}, r, s, 0)$ is conformal.
\end{proposition}
\begin{proof}
Define $\tilde{f}$ by the formula $\tilde{f}(X)=f\left(\frac{1}{r-1}(X+\gamma)  \right)$. Now consider the algebra epimomorphism $\phi: \K\langle d, u, h \rangle\rightarrow L(\tilde{f}, r, s, 0)$ defined on the free 
$\K$-algebra on free generators $d$, $u$, $h$ by:
\begin{equation*}
\phi (d)=d,\q\q \phi (u)=u\q\q\text{and}\q\q\phi (h)=\frac{1}{r-1}(h+\gamma).
\end{equation*}

Using the relations in $L(\tilde{f}, r, s, 0)$ and the definition of $\tilde{f}$ we find that:
\begin{eqnarray*}
\phi (dh-rhd+\gamma d)&=&\frac{1}{r-1}d(h+\gamma)-\frac{r}{r-1}(h+\gamma)d+\gamma d\\
&=&\frac{1}{r-1}\left( (dh-rhd)+(1-r)\gamma d \right)+\gamma d\\
&=&\frac{1-r}{r-1}\gamma d+\gamma d=0;
\end{eqnarray*}
similarly, $\phi (hu-ruh+\gamma u)=0$; and finally
\begin{eqnarray*}
\phi (du-sud+f(h))&=&du-sud+f\left(\frac{1}{r-1}(h+\gamma)  \right)\\
&=&du-sud+\tilde{f}(h)=0.
\end{eqnarray*}
Therefore, by~(\ref{E:defGdua1})--(\ref{E:defGdua3}), $\phi$ induces an algebra epimorphism,  still denoted  $\phi$, $L(f, r, s, \gamma)\rightarrow L(\tilde{f}, r, s, 0)$. To conclude that this map is an isomorphism it is enough to proceed similarly and define an algebra map $\psi: L(\tilde{f}, r, s, 0)\rightarrow L(f, r, s, \gamma)$ satisfying:
\begin{equation*}
\psi (d)=d,\q\q \psi (u)=u\q\q\text{and}\q\q\psi (h)=(r-1)h-\gamma.
\end{equation*}
The maps $\phi$ and $\psi$ are mutual inverses.

For $g\in\K[X]$ define $\tilde{g}$ by $\tilde{g}(X)=g\left(\frac{1}{r-1}(X+\gamma)  \right)$. The last statement follows because the equations $f(X)=sg(X)-g(rX-\gamma)$ and $\tilde{f}(X)=s\tilde{g}(X)-\tilde{g}(rX)$ are equivalent in $\K[X]$.
\end{proof}

In view of Lemma~\ref{L:conf:gamma0} and Proposition~\ref{P:conf:gamma0rnot1}, it remains to determine when $L(f, 1, s, \gamma)$ is conformal, which is what we do next.

\begin{proposition}\label{P:conf:ris1}
$L(f, 1, s, \gamma)$ is conformal in all cases except the case $L(f, 1, 1, 0)$ with $f\neq 0$.
\end{proposition}
\begin{proof}
Suppose $s\neq 1$. Then by~\cite[Lem.\ 2.8]{CS04} $L(f, 1, s, \gamma)$ is conformal ($s\neq r^{i}$ for all $i$). Also, $L(f, 1, s, \gamma)$ is conformal if $f=0$. Hence, if $L(f, 1, s, \gamma)$ is not conformal then $s=1$ and $f\neq 0$. 

Suppose first that $s=1$, $f\neq 0$ and $\gamma=0$. Then the conformality condition becomes $f(X)=g(X)-g(X)=0$, so indeed $L(f, 1, 1, 0)$ is not conformal for $f\neq 0$. It remains to show that $L(f, 1, 1, \gamma)$ is conformal if $f\neq 0$ and $\gamma\neq 0$. This amounts to showing that the linear map 
$\K[X]\rightarrow\K[X]$ defined by $g(X)\mapsto g(X)-g(X-\gamma)$ is onto. A routine induction on $n$ shows that $X^n$ is in the image of this map for all $n\geq 0$, so the map is indeed onto .
\end{proof}

\begin{remark}
The notion of conformality is not invariant under isomorphism. For example, there is an isomorphism $L(X, 1, 2, 1)\rightarrow L(X+1, 2, 1, 0)$, taking $d$ to $d$, $u$ to $u$ and $h$ to $ud+h+1$. Nevertheless, $L(X, 1, 2, 1)$ is conformal, by Proposition~\ref{P:conf:ris1}, whereas by 
Lemma~\ref {L:conf:gamma0}, $L(X+1, 2, 1, 0)$ is not conformal.
\end{remark}


\subsection{The $\Z$-grading}\label{SS:grad}

Given the defining relations~(\ref{E:defGdua1})--(\ref{E:defGdua3}),  there is a $\Z$-grading of $L$ obtained by assigning to the generators $d$, $u$ and $h$ the degrees $-1$, $1$ and $0$, respectively~\cite[Sec.\ 4]{CS04}. We thus get a decomposition $L=\bigoplus_{i\in\Z}L_{i}$ of $L$ into homogeneous subspaces. Whenever $rs\neq 0$ these are easy to describe, either by using the isomorphism $L\simeq D(\sigma, a)$ of Lemma~\ref{L:noeth:gwa}, or by invoking~\cite[Prop.\ 4.1]{CS04}:

\begin{proposition}\label{P:grad:Li}
Assume $rs\neq 0$. Then $L_{0}=D=\K[h, a]$ is the commutative polynomial algebra generated by $h$ and $a=ud$, $L_{-i}=Dd^{i}=d^{i}D$ and $L_{i}=Du^{i}=u^{i}D$, for $i>0$.
\end{proposition}

This result  has some interesting consequences, as the next Corollary shows. We recall the reader that an element $t$ of a ring $R$ is said to be normal if $tR=Rt$.

\begin{corollary}\label{C:grad:homrels}
Assume $rs\neq 0$ and let $t\in L_{i}$, for some $i\in\Z$. Recall the automorphism $\sigma$ of $D$ defined just before Lemma~\ref{L:noeth:gwa} by $\sigma(h)=rh-\gamma$ and 
$\sigma(a)=sa-f(h)$. Then:
\begin{enumerate}
\item $pt=t\sigma^{i}(p)$ for all $p\in D$;
\item If $t$ is also assumed to be normal then there exist $\lambda, \mu\in\K^{*}$ such that $td=\lambda dt$ and $tu=\mu ut$.
\end{enumerate}
\end{corollary}

\begin{proof}
Let $t\in L_{i}$ and suppose $i\geq 0$. By Proposition~\ref{P:grad:Li}, there exists $b\in D$ such that $t=bu^{i}$. Then, for $p\in D$,
$
pt=pbu^{i}=bpu^{i}=bu^{i}\sigma^{i}(p)=t\sigma^{i}(p)
$.

Assume further that $t$ is normal and nonzero. There exist $\xi, \zeta\in L$ satisfying $td=\xi t$ and $dt=t\zeta$. Since $td$ and 
$dt$ are homogeneous of degree $i-1$, the elements $\xi$ and  $\zeta$ must also be homogeneous of degree $-1$, as $L$ is a domain. Thus, there are $p, q\in D$ so that $\xi=pd$ and $\zeta=qd$. The computation
\begin{equation*}
td=\xi t=pdt=pt\zeta=ptqd=t\sigma^{i}(p)qd
\end{equation*}
implies $\sigma^{i}(p)q=1$. So $\sigma^{i}(p)$ and $q$ are units of $D$. In particular, 
$p=\lambda\in\K^{*}$ as $\sigma$ is an automorphism of $D$. Then, $td=\lambda dt$. Similarly, $tu=\mu ut$.

The proof of the case $i<0$ is symmetric.
\end{proof}


\section{Automorphisms of generalized down-up algebras}\label{S:aut}

In this section we will describe the group of automorphisms of the Noetherian, conformal generalized down-up algebras $L(f, r, s, \gamma)$, under the additional assumption that the parameter $r$ is not a root of unity. As $r\neq 1$, it can be assumed by Proposition~\ref{P:conf:gamma0rnot1}  that $\gamma=0$ and that there is $g\in\K[X]$ satisfying $f(X)=sg(X)-g(rX)$. Recalling Lemma~\ref{L:conf:gamma0}, it can be further  assumed that $\mathrm{supp}\, (f)=\mathrm{supp}\, (g)$, so that $g$ is uniquely determined by $f$; in particular, $\deg (f)=\deg (g)$. Hence, for the remainder of Section~\ref{S:aut} we assume $\gamma=0$. 

It will be more convenient for us to use the generalized Weyl algebra approach. Let $a=ud$ and $k=a-g(h)$. Then $h$ and $k$ are generators of the polynomial algebra  $D=\K[h, a]$ and the automorphism 
$\sigma$ acts on $k$ by
\begin{equation*}
\sigma (k)=\sigma (a-g(h))=sa-f(h)-g(rh)=sa-sg(h)=sk.
\end{equation*}
Therefore, $L$ is presented as the generalized Weyl algebra $D(\sigma, k+g(h))$, where $D=\K[h, k]$ and $\sigma$ is the automorphism of $D$ defined by $\sigma (h)=rh$, $\sigma (k)=sk$. The relations are thus:
\begin{equation}\label{E:aut:gwa:gdua1}
xp(h, k)=p(rh, sk)x, \q\q p(h, k)y=yp(rh, sk),  \q\q \text{for all $p\in D$,}
\end{equation}
\begin{equation}\label{E:aut:gwa:gdua2}
yx=k+g(h), \q\q xy=sk+g(rh).
\end{equation}
The parameters $r, s\in\K$ satisfy $rs\neq 0$ and $r^{i}=1\iff i=0$. The connection between $D(\sigma, k+g(h))$ and $L(f, r, s, 0)$ is given by the isomorphism $h\mapsto h$, $k\mapsto ud-g(h)$, $x\mapsto d$, $y\mapsto u$.

\subsection{The center of $L$}\label{SS:cent}

Define $\epsilon\in\Z$ and $\tau\in\N$ by
\begin{align*}
\tau =& \min \{ i>0\mid s^{i}=r^{j}\q \text{for some $j\in\Z$} \}\q\q \text{and}\q\q r^{\epsilon}=s^{\tau}\\
& \text{if}\q \{ i>0\mid s^{i}=r^{j}\q \text{for some $j\in\Z$} \}\neq\emptyset , \\
\tau=& 0=\epsilon\q \text{otherwise}.
\end{align*}
Since $r$ is not a root of unity, $\epsilon$ is uniquely defined.

The next lemma is a routine exercise.

\begin{lemma}\label{L:cent:rs1} 
Let $\delta, \eta\in\Z$. Then
$$
r^{\delta}s^{\eta}=1\iff (\delta , \eta)=\lambda(-\epsilon, \tau)\q \text{for some $\lambda\in\Z$.}
$$
\end{lemma}

\begin{proposition}
The center $\Z (L)$ of $L$ is $\K$ if either $\tau=0$ or $\epsilon>0$, and it is the polynomial algebra $\K[h^{-\epsilon}k^{\tau}]$ if $\tau>0$ and $\epsilon\leq 0$.
\end{proposition}
\begin{proof}
As the canonical generators of $L$ are homogeneous with respect to the $\Z$-grading defined 
in~\ref{SS:grad}, it follows that $\Z (L)$ is graded. This means that if $z$ is central and 
$z=z_{i_{1}}+\cdots +z_{i_{m}}$ is the decomposition of $z$ into homogeneous components, then each of the $z_{i_{j}}$ is itself central. So we just need to determine $\Z(L)\cap L_{i}$ for all $i\in\Z$.

Let $i\leq 0$ and take $px^{i}\in\Z(L)\cap L_{i}$, with   $p$ a nonzero element of $D$. Then, since $L$ is a domain and $r$ is not a root of $1$, the computation
\begin{align*}
0=hpx^{i}-px^{i}h=hpx^{i}-r^{i}phx^{i}=(1-r^{i})hpx^{i}
\end{align*}
implies that $i=0$. The situation is identical if we take $i\geq 0$; hence $\Z (L)\subseteq D$.

Now take $p=p(h, k)\in D$. Again we compute:
\begin{equation*}
xp(h, k)-p(h, k)x=(p(rh, sk)-p(h, k))x,
\end{equation*}
and likewise for $y$. Therefore, as $p$ commutes with $h$ and $k$, $p$ is central if and only if $p(rh, sk)=p(h, k)$. Write $p=\sum a_{ij}h^{i}k^{j}$. Thus $p\in\Z (L)$ precisely when $r^{i}s^{j}=1$ whenever $a_{ij}\neq 0$. In view of Lemma~\ref{L:cent:rs1}, this condition means that $(i, j)=\lambda(-\epsilon, \tau)$ for some $\lambda\in\Z$.

If $\tau>0$ and $\epsilon\leq 0$ then $\lambda\geq 0$ as $\lambda\tau=j\geq 0$ and thus 
$h^{i}k^{j}=\left(  h^{-\epsilon}k^{\tau}\right)^{\lambda}$. In this case $p\in\K[h^{-\epsilon}k^{\tau}]$ and $\Z[L]=\K[h^{-\epsilon}k^{\tau}]$. Otherwise either $\tau, \epsilon>0$ or $\tau=0=\epsilon$. In the first of these cases $\lambda$ must be zero and $(i, j)=(0, 0)$; in the second case $(i, j)=(0, 0)$ as well. Thus 
$\Z (L)=\K$.
\end{proof}


\subsection{The normal elements of $L$}\label{SS:norm}

We start out by classifying the normal elements of $L$ of degree zero.

\begin{lemma}\label{L:norm:norminD}
Let $p\in D$ be a nonzero normal element of $L$. Write $p(h, k)=h^{\alpha}k^{\beta}q(h, k)$, with   
$\alpha, \beta\in\N$ and $q\in D$  not a multiple of $h$ or $k$. Then:
\begin{enumerate}
\item $q\in\K^{*}$, if $\tau=0$;
\item $q\in\Z (L)$ and $q$ has a nonzero constant term as a polynomial in $\left( h^{-\epsilon}k^{\tau}\right)$, if $\tau>0$ and $\epsilon\leq 0$;
\item $q=\sum_{i=0}^{l}d_{i}\left( h^{\epsilon} \right)^{l-i}\left( k^{\tau} \right)^{i}$ with $l\geq 0$, $d_{i}\in\K$ and $d_{0}, d_{l}\neq 0$, if $\tau, \epsilon>0$.
\end{enumerate}
\end{lemma}
\begin{proof}
Since $h$ and $k$ are themselves normal and $L$ is a domain, it follows that $q$ is normal, and nonzero. Write $q(h, k)=\sum_{i=0}^{l}q_{i}(h)k^{i}$ with $l\geq 0$, $q_{i}(h)\in\K[h]$ and $q_{l}(h)\neq 0$. 

By Corollary~\ref{C:grad:homrels} there exists $\lambda\in\K^{*}$ such that $xq=\lambda qx$. Thus
\begin{equation*}
\lambda \left( \sum_{i=0}^{l}q_{i}(h)k^{i} \right) x=\lambda qx=xq=\sigma(q)x=\left( \sum_{i=0}^{l}s^{i}q_{i}(rh)k^{i} \right) x
\end{equation*}
and we conclude that 
\begin{equation}\label{E:aut:norm:qi}
\lambda q_{i}(h)=s^{i}q_{i}(rh),\q\q \text{for all $0\leq i\leq l$.}
\end{equation}
Now fix $i$ and write $q_{i}(h)=\sum_{j}\alpha_{j}h^{j}$. If $\alpha_{j}\neq 0$ then~(\ref{E:aut:norm:qi}) implies that $r^{j}=\lambda s^{-i}$. As $r$ is not a root of $1$, $j=n_{i}$ is determined by $i$ and 
$q_{i}(h)=a_{i}h^{n_{i}}$, for some $a_{i}\in\K$. 

So far we have 
\begin{equation*}
q(h, k)=\sum_{i=0}^{l}a_{i}h^{n_{i}}k^{i}.
\end{equation*}
As $q$ is not a multiple of $k$, it must be that $a_{0}\neq 0$ and consequently $\lambda=r^{n_{0}}$. Therefore, for every $i$ such that $a_{i}\neq 0$, we have $r^{n_{i}-n_{0}}s^{i}=1$. By 
Lemma~\ref{L:cent:rs1}, there is $T_{i}\in\Z$ such that 
\begin{equation}\label{E:aut:norm:nin0i}
n_{i}-n_{0}=-\epsilon T_{i}\q\q \text{and}\q\q i=\tau T_{i}. 
\end{equation}
To finish our argument, we just need to distinguish between the three possibilities for the pair $(\epsilon, \tau)$ and use~(\ref{E:aut:norm:nin0i}).

If $\tau=0$ then necessarily $i=0$ and $q=a_{0}h^{n_{0}}$. Also, $n_{0}$ must be zero so that $q$ is not a multiple of $h$. This establishes (a).

If $\tau>0$ and $\epsilon\leq 0$ then $T_{i}\geq 0$ and $n_{i}=n_{0}-\epsilon T_{i}\geq n_{0}$. Hence $n_{0}=0$, to ensure that  $q$ is not a multiple of $h$, and 
$h^{n_{i}}k^{i}=\left( h^{-\epsilon}k^{\tau}\right)^{T_{i}}\in\Z(L)$. So $q$ is indeed central and the constant term of $q$ when written as a polynomial in  $\left( h^{-\epsilon}k^{\tau}\right)$ must be nonzero, or otherwise $q$ would be a multiple of $k$.

Let us analyze the final case with $\tau>0$ and $\epsilon> 0$. Again, $T_{i}\geq 0$ 
by~(\ref{E:aut:norm:nin0i}). Moreover, there is $0\leq i\leq l$ such that $a_{i}\neq 0$ and $n_{i}=0$, by the condition that $q$ is not a multiple of $h$. It follows from~(\ref{E:aut:norm:nin0i}) that $\epsilon$ divides $n_{0}$, say $n_{0}=\epsilon m$. Hence, for all $i$ such that $a_{i}\neq 0$, 
$h^{n_{i}}k^{i}=\left( h^{\epsilon}\right)^{m-T_{i}}\left( k^{\tau}\right)^{T_{i}}$. In particular, $0\leq T_{i}\leq m$ and $q$ can be written as 
\begin{equation*}
q(h, k)=\sum_{i=0}^{m}d_{i}\left( h^{\epsilon}\right)^{m-i}\left( k^{\tau}\right)^{i}
\end{equation*}
with $d_{i}\in\K$ and $d_{0}, d_{m}\neq 0$, ensuring $q$ is neither a multiple of $k$ nor of $h$.

\end{proof}

Our next step in describing the monoid of normal elements of $L$ is to determine when $x^{n}$ and 
$y^{n}$ are normal.

\begin{lemma}\label{L:norm:normxnyn}
The following conditions are equivalent, for $n\geq 1$:
\begin{enumerate}
\item $x^{n}$ is normal;
\item Either $f=0$ or $f(X)=\mu (s-r^{m})X^{m}$ for some $\mu\in\K^{*}$ and some $m\geq 0$ so that $\epsilon=\tau m$ and $\tau$ divides $n$;
\item $yx^{n}=s^{-n}x^{n}y$;
\item $xy^{n}=s^{n}y^{n}x$;
\item $y^{n}$ is normal.
\end{enumerate}
In particular, if $x^{n}$ is normal then either $f=0$ or $\tau, n>1$.
\end{lemma}
\begin{proof}
The algebra antiautomorphism interchanging $x$ and $y$ referred  to at the end of Section~\ref{SS:prel} proves the equivalence of statements (a) and (e) and of statements (c) and (d). It remains to show the series of implications: $(a)\implies (b)\implies (c)\implies (a)$.

So assume $x^{n}$ is normal, for some $n\geq 1$. By Corollary~\ref{C:grad:homrels} there exists 
$\lambda\in\K^{*}$ such that $yx^{n}=\lambda x^{n}y$. Hence, 
\begin{equation*}
(k+g(h))x^{n-1}=yx^{n}=\lambda x^{n}y=\lambda x^{n-1}(sk+g(rh))= \lambda (s^{n}k+g(r^{n}h))x^{n-1},
\end{equation*}
from which the following equality in $D$ is deduced: $k+g(h)=s^{n}\lambda k+\lambda g(r^{n}h)$. Comparing coefficients of $k$ in this last equation yields $s^{n}\lambda=1$ and $\lambda g(r^{n}h)=g(h)$. Thus,
\begin{equation}\label{E:L:norm:normxn:g}
g(r^{n}h)=s^{n}g(h).
\end{equation}
Since $r$ is not a root of $1$, there exist $\mu\in\K$ and $m\in\N$ so that $g(X)=\mu X^{m}$ and $f(X)=\mu (s-r^{m})X^{m}$.

If $\mu=0$ then $0=g=f$. Otherwise, assume $\mu\neq 0$. Then condition~(\ref{E:L:norm:normxn:g}) translates to $r^{nm}=s^{n}$. By Lemma~\ref{L:cent:rs1} we have
\begin{equation*}
nm=\epsilon T\q \text{and} \q n=\tau T,\q\q \text{for some $T\in\Z$.}
\end{equation*}
In particular, $\tau$ divides $n$ and $\epsilon=\tau m$. Notice that, in this case, we cannot have $n=1$, as this would imply $\tau=1$, $\epsilon=m$ and $f(X)=0$, contrary to our supposition. Hence both integers $n$ and $\tau$ must be greater than $1$, if $f\neq 0$.

Now assume (b) holds. If $f=0$ then  $g=0$ and by~(\ref{E:aut:gwa:gdua2}), $xy=sk=syx$. It follows 
that  $x^{n}y=s^{n}yx^{n}$. Instead, suppose $g(X)=\mu X^{m}$ for $\mu\in\K^{*}$ and $m\in\N$. As $\tau$ divides $n$, it is enough to show that $x^{\tau}y=s^{\tau}yx^{\tau}$. This is indeed a true statement as, by hypothesis, $r^{\tau m}=r^{\epsilon}=s^{\tau}$:
\begin{align*}
x^{\tau}y&=x^{\tau-1}(sk+g(rh))=(s^{\tau}k+g(r^{\tau}h))x^{\tau-1}\\
&=(s^{\tau}k+r^{\tau m}g(h))x^{\tau-1}=s^{\tau}(k+g(h))x^{\tau-1}=s^{\tau}yx^{\tau}.
\end{align*}
In either case, (c) holds.

Finally, if (c) holds then clearly $x^{n}$ is normal, by~(\ref{E:aut:gwa:gdua1}).
\end{proof}

We are finally ready to describe all normal elements of $L$.

\begin{proposition}\label{P:aut:norm:allnorm}
The normal elements of $L$ are the elements of the form $p(h, k)x^{n}$ and $p(h, k)y^{n}$, with 
$n\geq 0$ and $p(h, k)\in D$ such that  $p(h, k)$, $x^{n}$ and $y^{n}$ are normal. 
\end{proposition}
\begin{proof}
Since the product of normal elements is normal, it is clear that all of the indicated elements are normal. Conversely, let $0\neq t\in L$ be normal. Write $t=\sum_{j\in J}t_{j}$ with $J$ a finite nonempty subset of $\Z$ and $0\neq t_{j}\in L_{j}$.

\emph{Claim}: $t$ is homogeneous, i.e., $|J|=1$.

\emph{Proof of claim}: By the normality of $t$, $ht=tt'$ for some $t'\in L$. The $\Z$-grading of $L$, together with the fact that $L$ is a domain, imply that $t'\in D$. Note that by 
Corollary~\ref{C:grad:homrels}(a), $ht_{j}=r^{j}t_{j}h$ for all $j\in J$. Thus, 
\begin{equation*}
\sum_{j\in J}r^{j}t_{j}h=\sum_{j\in J}ht_{j}=ht=tt'=\sum_{j\in J}t_{j}t' .
\end{equation*}
Using again the $\Z$-grading and the fact that $L$ is a domain, we infer that  $r^{j}h=t'$,  for all $j\in J$. So, as claimed, $|J|=1$ because $r$ is not a root of unity.

We can assume, without loss of generality, that $t=p(h, k)x^{n}\in L_{-n}$, $n\in\N$, the case $t\in L_{n}$ being symmetric. As $t$ is homogeneous, Corollary~\ref{C:grad:homrels}(b) can be invoked to guarantee the existence of $\lambda\in\K^{*}$ satisfying $xt=\lambda tx$. Working out this equation in $L$ leads to the equivalent equation $p(rh, sk)=\lambda p(h, k)$ in $D$. Hence 
$xp(h, k)=\lambda p(h, k)x$ and $yp(h, k)=\lambda^{-1} p(h, k)y$, showing that $p(h, k)$ is normal in 
$L$.

Finally, to prove that $x^{n}$ is normal we use Corollary~\ref{C:grad:homrels}(b) once more: there is 
$\mu\in\K^{*}$ such that $ty=\mu yt$. So
\begin{equation*}
p(h, k)x^{n}y=ty=\mu yt=\mu yp(h, k)x^{n}=\mu\lambda^{-1} p(h, k)yx^{n}
\end{equation*}
and $x^{n}y=\mu\lambda^{-1} yx^{n}$. Thus $x^{n}$ is normal as well, by~(\ref{E:aut:gwa:gdua1}).
\end{proof}

Combining Proposition~\ref{P:aut:norm:allnorm} with Lemma~\ref{L:norm:norminD} and Lemma~\ref{L:norm:normxnyn}, we obtain a complete description of all normal elements of $L$.
Before we end this section, we record a straightforward, yet useful, result, which holds  in any domain if we replace $\K^{*}$ by its group of units. 

\begin{lemma}\label{L:aut:norm:prodnormpri}
Assume $t, v\in L$ are nonzero normal elements whose product  generates a prime ideal of $L$. Then either $t\in\K^{*}$ or $v\in\K^{*}$.
\end{lemma}

\subsection{Some properties of the automorphisms of $L$}\label{SS:propautL}

In this section we gather some general information about the automorphisms of $L$. We denote the group of algebra automorphisms of $L$ by $\aut$.

Recall that a proper ideal $P$ of a ring $R$ is said to be completely prime if the factor ring $R/P$ is a domain. In particular, completely prime ideals are prime.

\begin{lemma}\label{L:propautL:phih}
Let $\phi\in\aut$. Then $\phi (h)$ is a normal element of $L$ which generates a completely prime ideal of $L$.
\end{lemma}
\begin{proof}
It needs to be shown that $h$ is normal and generates a completely prime ideal of $L$, as these two properties are invariant by automorphisms. The first one is clear, as $h$ is central in $D$ and $\sigma (h)=rh$ (see~(\ref{E:aut:gwa:gdua1})). To prove that $\langle h\rangle$ is completely prime we need to argue that the factor algebra $L/\langle h\rangle$ is a domain. 

By relations~(\ref{E:defGdua1})--(\ref{E:defGdua3}), with $\gamma=0$, $L/\langle h\rangle$ is the algebra generated by $\bar{x}$ and $\bar{y}$, subject only to the relation $\bar{x}\bar{y}-s\bar{y}\bar{x}+f_{{}_{0}}=0$, $f_{{}_{0}}$ being the constant term of the polynomial $f$. There are four possibilities, depending on the scalars $s$  and $f_{{}_{0}}$. If $s=1$ we are in the classical setting and the factor algebra is either a commutative polynomial algebra  in two variables ($f_{{}_{0}}=0$) or the first Weyl algebra over $\K$ ($f_{{}_{0}}\neq 0$). If $s\neq 1$ we are in the quantum setting. Recalling that we are assuming also $s\neq 0$, the factor algebra is either a  quantum plane ($f_{{}_{0}}=0$) or the first quantum Weyl algebra ($f_{{}_{0}}\neq 0$). Any of these four algebras is a domain, so the ideal $\langle h\rangle$ is completely prime.
\end{proof}

\begin{lemma}\label{L:propautL:phihnotxn}
Assume $n\geq 1$. Then $x^{n}$ (resp.\ $y^{n}$) is normal and generates a completely prime ideal of $L$ if and only if $n=1$ and $f=0$.
\end{lemma}
\begin{proof}
If $f=0$ then $x$ is normal, by Lemma~\ref{L:norm:normxnyn}. In this case, the ideal $\langle x\rangle$ is completely prime since the factor algebra $L/\langle x\rangle$ is easily seen to be a quantum plane, generated by $\bar{h}$ and $\bar{y}$,  satisfying  the relation $\bar{h}\bar{y}=r\bar{y}\bar{h}$.

Conversely, assume that $x^{n}$ is normal and that the ideal $\langle x^{n}\rangle=x^{n}L=Lx^{n}$ is completely prime. Suppose, by way of contradiction, that $n>1$. Then, since $xx^{n-1}\in Lx^{n}$, it must be that either $x\in Lx^{n}$ or $x^{n-1}\in Lx^{n}$. In any case, 
$x^{n-1}\in Lx^{n}$, as $n-1\geq 1$, and there is $v\in L$ so that  $1=vx$ because $L$ is a domain. Similarly, there is $v'\in L$ so that $1=xv'$. This is a contradiction because $x$ is not a unit in $L$, by~Lemma~\ref{L:noeth:units}. Therefore $n=1$. By 
Lemma~\ref{L:norm:normxnyn}, $f=0$.
\end{proof}

\begin{lemma}\label{L:propautL:eigenv}
Let $t, v\in D\setminus\{ 0 \}$. Suppose $xt=\lambda tx$ and $xv=\mu vx$, for some 
$\lambda, \mu\in\K^{*}$. If there is $\phi\in\aut$ such that $\phi (t)=v$ then $\langle\langle \lambda\rangle\rangle=\langle\langle \mu\rangle\rangle$, where $\langle\langle \xi \rangle\rangle$ denotes the subgroup of $\K^{*}$ generated by $\xi$.
\end{lemma}
\begin{proof}
Replacing $\phi$ by $\phi^{-1}$, it is enough to show that $\langle\langle \lambda\rangle\rangle\subseteq\langle\langle \mu\rangle\rangle$.

If we multiply both sides of equation $xv=\mu vx$ on the left by $y$, observe that $yx\in D$ commutes with $v$ and use the fact that $L$ is a domain, we obtain $vy=\mu yv$. Therefore,
\begin{equation}\label{E:propautL:eigenv}
vl=\mu^{i}lv, \q\q \text{for all $l\in L_{i}$.}
\end{equation}
Let us write $\phi (x)=\sum_{j} x_{j}$, with $x_{j}\in L_{j}$. Applying $\phi$ to equation $xt=\lambda tx$ and using~(\ref{E:propautL:eigenv}), we deduce the following: 
$\sum_{j}x_{j}v=\sum_{j}\lambda \mu^{j}x_{j}v$. By the $\Z$-grading, $\lambda \mu^{j}=1$ for all $j\in\Z$ such that $x_{j}\neq 0$. Since $\phi (x)\neq 0$, there is $i\in\Z$ with $\lambda=\mu^{i}$. Thus $\lambda\in\langle\langle \mu\rangle\rangle$ and $\langle\langle \lambda\rangle\rangle\subseteq\langle\langle \mu\rangle\rangle$, as desired.

\end{proof}

\subsection{The  automorphisms of $L$}\label{SS:grautL}

Now we  describe the algebra automorphisms of $L$ in detail. The results of this section will be combined in the next section to determine  the group $\aut$.

Recall that $f(X)=sg(X)-g(rX)$. In case $f\neq 0$, we define a nonnegative integer $\rho$ by
\begin{equation*}
\rho=\gcd \{ \deg (f)-i\mid i\in\mathrm{supp}\, (f) \}
\end{equation*}
if $\{ \deg (f)-i\mid i\in\mathrm{supp}\, (f) \}\neq \{ 0\}$, and $\rho=0$ otherwise. We do not define $\rho$ if $f=0$.

\begin{lemma}\label{L:grautL:rho}
Let $\lambda, \mu\in\K$ with $\lambda\neq 0$. Then $f(\lambda X)=\mu f(X)$ $\iff$ either $f=0$, or $\lambda^{\rho}=1$ and $\lambda^{\deg (f)}=\mu$.
\end{lemma}
\begin{proof}
Notice that $f(\lambda X)=\mu f(X)\iff \lambda^{i}=\mu$ for all $i\in \mathrm{supp}\, (f)$. If $f=0$ this condition is clearly satisfied, so assume $\lambda^{\rho}=1$ and $\lambda^{\deg (f)}=\mu$. Given 
$i\in \mathrm{supp}\, (f)$ we have, by the definition of $\rho$, $\lambda^{\deg (f)-i}=1$. Hence, 
$\mu=\lambda^{\deg (f)}=\lambda^{i}$.

Conversely, assume  that $f(\lambda X)=\mu f(X)$ and $f\neq 0$. Then, in particular, $\lambda^{\deg (f)}=\mu$. If $\rho=0$ there is nothing else to prove. Assume $\rho\neq 0$. By hypothesis, $\lambda^{i}=\mu=\lambda^{\deg{f}}$ for all $i\in \mathrm{supp}\, (f)$. As $\lambda\neq 0$ we have 
$\lambda^{\deg (f)-i}=1$ for all $i\in\mathrm{supp}\, (f)$. Hence $\lambda^{\rho}=1$.
\end{proof}

Consider the following subgroup of $\aut$:
\begin{equation*}
\h=\{ \phi\in\aut \mid \phi (h)=\lambda h, \  \text{for some $\lambda\in\K^{*}$} \}.
\end{equation*}

\begin{lemma}\label{L:grautL:someh}
The following define elements of $\h$:
\begin{enumerate}
\item If $f=0$ there is a unique 
$\phi_{(\alpha, \beta, \gamma)}\in\h$ defined on the generators  by 
$\phi_{(\alpha, \beta, \gamma)}(h)=\alpha h$, $\phi_{(\alpha, \beta, \gamma)}(x)=\beta x$, $\phi_{(\alpha, \beta, \gamma)}(y)=\gamma y$, for any $(\alpha, \beta, \gamma)\in\left(\K^{*}\right)^{3}$;
\item If $f\neq 0$ there is a unique 
$\phi_{(\alpha, \beta)}\in\h$ defined on the generators  by 
$\phi_{(\alpha, \beta)}(h)=\alpha h$, $\phi_{(\alpha, \beta)}(x)=\beta x$, 
$\phi_{(\alpha, \beta)}(y)=\beta^{-1}\alpha^{\deg (f)}y$, for any $(\alpha, \beta )\in\left(\K^{*}\right)^{2}$ such that $\alpha^{\rho}=1$.
\end{enumerate}
\end{lemma}
\begin{proof}
The uniqueness is clear, as $L$ is generated by $h$, $x$ and $y$. To prove the existence, we need to check that relations~(\ref{E:defGdua1})--(\ref{E:defGdua3}), with $\gamma=0$ and $d$ (resp.\ $u$) replaced by $x$ (resp.\ $y$), are preserved when we define the homomorphism on the free algebra on generators $h$, $x$ and $y$, and to argue  the existence of an inverse. 

If $f=0$ then the relations are homogeneous in the generators and hence $\phi_{(\alpha, \beta, \gamma)}$ is indeed an automorphism, with inverse $\phi_{(\alpha^{-1}, \beta^{-1}, \gamma^{-1})}$.

Now consider the case $f\neq 0$. Relations $xh=rhx$ and $hy=ryh$ are trivial to check. When we apply 
$\phi_{(\alpha, \beta)}$ to  $xy-syx+f(h)$ we obtain $\alpha^{\deg (f)}(xy-syx)+f(\alpha h)$. Thus, we must have $f(\alpha h)=\alpha^{\deg (f)}f(h)$ for $\phi_{(\alpha, \beta)}$ to be a homomorphism of $L$. By Lemma~\ref{L:grautL:rho}, this is indeed the case, as we have the additional restriction that $\alpha^{\rho}=1$. Furthermore, $\phi^{-1}_{(\alpha, \beta)}=\phi_{(\alpha^{-1}, \beta^{-1})}$.
\end{proof}

Our next result describes the group $\h$.

\begin{proposition}\label{P:grautL:allh}
\begin{enumerate}
\item If $f=0$ then $\h=\{ \phi_{(\alpha, \beta, \gamma)}\mid (\alpha, \beta, \gamma)\in\left(\K^{*}\right)^{3} \}\simeq \left(\K^{*}\right)^{3}$, with $\phi_{(\alpha, \beta, \gamma)}$ as given in Lemma~\ref{L:grautL:someh}.
\item If $f\neq 0$ then $\h=\{ \phi_{(\alpha, \beta)}\mid (\alpha, \beta)\in\left(\K^{*}\right)^{2} \  
\text{and} \  \   \alpha^{\rho}=1 \}$, with $\phi_{(\alpha, \beta)}$ as given in Lemma~\ref{L:grautL:someh}. Consequently, $\h\simeq\left(\K^{*}\right)^{2}$ if $\rho=0$ and 
$\h\simeq \Z/\rho\Z \times\K^{*}$ if $\rho>0$.
\end{enumerate}
\end{proposition}
\begin{proof}
Let $\phi\in\aut$ with $\phi (h)=\alpha h$, for some $\alpha\in\K^{*}$. For $i\in\Z$ and $l\in L_{i}$, we have the relation $hl=r^{i}lh$. Upon applying $\phi$ to this relation and dividing by $\alpha$ we obtain the relation $h\phi (l)=r^{i}\phi (l)h$. Given the $\Z$-grading and since $L$ is a domain and $r$ is not a root of unity, it is routine to conclude that $\phi (L_{i})\subseteq L_{i}$. Moreover, as $\phi$ is onto it follows that $\phi (L_{i})=L_{i}$, for all $i\in\Z$.

Take $t$, $v\in D$ so that $\phi (x)=tx$ and $\phi (vx)=x$. Then, $x=\phi (vx)=\phi (v)\phi (x)=\phi (v)tx$ and thus $\phi (v)t=1$. Since both $\phi (v)$ and $t$ are elements of $D$, the latter implies that $t$ is a unit. So $\phi (x)=\beta x$, for some $\beta\in\K^{*}$; similarly, $\phi (y)=\gamma y$, for some $\gamma\in\K^{*}$.

If $f=0$ then $\phi=\phi_{(\alpha, \beta, \gamma)}$, as described in Lemma~\ref{L:grautL:someh}(a), and the map $\left(\K^{*}\right)^{3}\rightarrow\h$, $(\alpha, \beta, \gamma)\mapsto \phi_{(\alpha, \beta, \gamma)}$ is a group isomorphism.

Now suppose $f\neq 0$. 
Applying $\phi$ to both sides of the relation $xy-syx+f(h)=0$ yields 
$\beta\gamma(xy-syx)+f(\alpha h)=0$, which is equivalent in $L$ to  $\beta\gamma f(h)=f(\alpha h)$. Since the elements $\left\{ h^{j} \right\}_{j\geq 0}$ are linearly independent over $\K$, we have 
$\beta\gamma f(X)=f(\alpha X)$. Thus, by Lemma~\ref{L:grautL:rho}, $\alpha^{\rho}=1$ and 
$\beta\gamma=\alpha^{\deg (f)}$. So $\phi=\phi_{(\alpha, \beta)}$. 

If $\rho=0$ then  $\alpha$ and $\beta\in\K^{*}$ are arbitrary and $\left(\K^{*}\right)^{2}\rightarrow\h$, $(\alpha, \beta)\mapsto \phi_{(\alpha, \beta)}$ is a group isomorphism. Otherwise, if $\rho\geq 1$, let 
$\xi\in\K$ be a primitive $\rho$-th root of unity. Then  the multiplicative group $\{ \alpha\in\K^{*}\mid \alpha^{\rho}=1 \}=\{ \xi^{i}\mid 0\leq i\leq \rho-1 \}$ is isomorphic to the additive group $\Z/\rho\Z$ of integers modulo $\rho$ and $\Z/\rho\Z\times\K^{*}\rightarrow\h$, $(i+\rho\Z, \beta)\mapsto \phi_{(\xi^{i}, \beta)}$ is a group isomorphism.
\end{proof}

Now we turn our attention to automorphisms of $L$ not necessarily fixing the ideal $\langle h\rangle$.

\begin{lemma}\label{L:grautL:someother}
Assume $\tau>0$ and $f(X)=\alpha X+\beta$ for some $\alpha, \beta\in\K$. The following define automorphisms of $L$:
\begin{enumerate}
\item If $\epsilon=1$ and $\alpha=0$  then there is a unique $\psi^{+}_{(\mu, \mu', \nu, \eta)}\in\aut$ so that $\psi^{+}_{(\mu, \mu', \nu, \eta)}(x)=\mu x$, 
$\psi^{+}_{(\mu, \mu', \nu, \eta)}(y)=\mu' y$ and $\psi^{+}_{(\mu, \mu', \nu, \eta)}(h)=\nu h+\eta k^{\tau}$, for all $(\mu, \mu', \nu, \eta)\in\left(\K^{*} \right)^{3}\times\K$ with $\beta (\mu\mu'-1)=0$. 

\item If $\tau=1$, $\epsilon=-1$ and $\alpha\neq 0$ then there is a unique $\psi^{-}_{(\mu, \nu )}\in\aut$ so that $\psi^{-}_{(\mu, \nu )}(x)=\mu y$, $\psi^{-}_{(\mu, \nu )}(y)=\frac{r\alpha\nu}{\mu(s-r)} x$ and $\psi^{-}_{(\mu, \nu )}(h)=\nu k$, for all $(\mu, \nu )\in\left(\K^{*} \right)^{2}$ with $\beta (\frac{r\alpha\nu}{s-r}-1)=0$.
\end{enumerate}
\end{lemma}

\begin{proof}
As in the proof of Lemma~\ref{L:grautL:someh}, the uniqueness is clear, and the existence follows from checking that relations~(\ref{E:defGdua1})--(\ref{E:defGdua3}) are preserved and from the construction of an inverse homomorphism.

Suppose first that $\epsilon=1$, $\alpha=0$  and the scalars $\mu$, $\mu'$, $\nu$, $\eta\in\K$ satisfy $\mu \mu'\nu\neq 0$ and 
$\beta (\mu\mu'-1)=0$. In this case, $f(X)=$ is a constant polynomial and 
\begin{align*}
\psi^{+}_{(\mu, \mu', \nu, \eta)}(xh-rhx) &= \mu \left(x (\nu h+\eta k^{\tau})-r(\nu h+\eta k^{\tau})x\right)\\
&=0,\q\q\q\text{as $r=s^{\tau}$;}\\[8pt]
\psi^{+}_{(\mu, \mu', \nu, \eta)}(hy-ryh) &= \mu' \left( (\nu h+\eta k^{\tau})y-ry(\nu h+\eta k^{\tau})\right)\\
&=0,\q\q\q\text{as $r=s^{\tau}$;}\\[8pt]
\psi^{+}_{(\mu, \mu', \nu, \eta)}(xy-syx+f(h)) &= \mu\mu' (xy-syx)+\beta \\
&=\mu\mu' \left(xy-syx+\beta \right) ,\q\q\q\text{as $\beta=\beta \mu\mu'$,} \\
&=0.
\end{align*}
Thus $\psi^{+}_{(\mu, \mu', \nu, \eta)}$ does indeed define an algebra endomorphism of $L$. Note also that $g(X)=\frac{\beta}{s-1}$ and thus $\psi^{+}_{(\mu, \mu', \nu, \eta)}(k)=\psi^{+}_{(\mu, \mu', \nu, \eta)}(yx-g(h))=\mu\mu' (yx-g(h))=\mu\mu' k$.
Finally, we can check that $\psi^{+}_{(\mu^{-1}, \mu'^{-1}, \nu^{-1}, \frac{-\eta}{(\mu\mu')^{\tau}\nu})}$ is the inverse of $\psi^{+}_{(\mu, \mu', \nu, \eta)}$.

Now assume $\tau=1$, $\epsilon=-1$, $\alpha\neq 0$, $(\mu, \nu )\in\left(\K^{*} \right)^{2}$ and $\beta (\frac{r\alpha\nu}{s-r}-1)=0$. Since $s=r^{-1}$ and $\beta=\frac{\alpha\beta\nu r}{s-r}$, we obtain 
\begin{align*}
\psi^{-}_{(\mu, \nu )}(xh-rhx) &= \mu\nu (yk-rky)=0;\\[8pt]
\psi^{-}_{(\mu, \nu )}(hy-ryh) &= \frac{r\alpha\nu^{2}}{\mu (s-r)} (kx-rxk)=0;\\[8pt]
\psi^{-}_{(\mu, \nu )}(xy-syx+f(h)) &= \frac{r\alpha\nu}{s-r} (yx-sxy)+f(\nu k)\\
&=\frac{r\alpha\nu}{s-r} ((1-r^{-2})k+g(h)-r^{-1}g(rh))+f(\nu k) \\
&=\frac{r\alpha\nu}{s-r} ((1-r^{-2})k+\frac{1-r^{-1}}{r^{-1}-1}\beta)+f(\nu k) \\
&=-(\alpha\nu k+\frac{\alpha\beta\nu r}{s-r})+f(\nu k)\\
&=0.
\end{align*}
So indeed $\psi^{-}_{(\mu, \nu )}$ defines an algebra endomorphism of $L$. It can be checked that $\psi^{-}_{(\mu, \nu )}(k)=\left(\frac{\alpha r}{s-r} \right)^{2}\nu h$ and that the  inverse of $\psi^{-}_{(\mu, \nu )}$ is $\psi^{-}_{(\frac{\mu(s-r)}{r\alpha\nu}, \left(\frac{s-r}{\alpha r}\right)^{2}\nu^{-1})}$.
\end{proof}

\begin{proposition}\label{P:grautL:another}
Suppose $\phi (h)=\lambda k$ for some $\phi\in\aut$  with $\lambda\in\K^{*}$. Then  
$s=r^{-1}$, there exists $(\alpha, \beta)\in\K^{*}\times\K$ so that $f(X)=\alpha X+\beta$, $\beta(\frac{r\alpha\lambda}{s-r}-1)=0$ and 
$\phi=\psi^{-}_{(\mu, \lambda )}$ for some $\mu\in\K^{*}$.
\end{proposition}

\begin{proof}
By Lemma~\ref{L:propautL:eigenv}, $r$ and $s$ generate the same multiplicative subgroup of $\K^{*}$ and so $s=r^{\pm 1}$, as $\langle\langle r\rangle\rangle$ is the infinite cyclic group. If we argue as in the proof of Proposition~\ref{P:grautL:allh}, we can deduce that, for any $i\in\Z$, $\phi (L_{i})=L_{i}$ if $s=r$ or $\phi (L_{i})=L_{-i}$ if $s=r^{-1}$. In either case, $\phi (D)=D$.

Suppose, by way of contradiction, that $s=r$. Then, again following the proof of 
Proposition~\ref{P:grautL:allh}, there exist $\mu, \mu'\in\K^{*}$ so that $\phi (x)=\mu x$ and 
$\phi (y)=\mu' y$. Upon applying $\phi$ to the relation $xy-syx+f(h)=0$ we obtain the equality 
$\mu\mu'(xy-syx)+f(\lambda k)=0$, which is equivalent to 
\begin{equation}\label{E:grautL:another:sr}
\mu\mu' f(h)=f(\lambda k).
\end{equation}
The left-hand-side of~(\ref{E:grautL:another:sr}) being a polynomial in $h$  whereas the right-hand-side is one in $k$ implies that $f$ is a constant polynomial, say $f=\beta\in\K$. Thus $g=\frac{\beta}{s-1}$  and equation~(\ref{E:grautL:another:sr}) yields $\beta=\beta\mu\mu'$. Hence, $\phi (k)=\phi (yx-\frac{\beta}{s-1})=\mu\mu' yx-\frac{\beta}{s-1}=\mu\mu' k$. This contradicts the injectivity of $\phi$, as we would have 
$\phi (\lambda k)=\mu\mu'\lambda k=\phi (\mu\mu' h)$, with $\lambda\mu\mu'\neq 0$. So indeed 
$s=r^{-1}$.

As before, given that $\phi (L_{\pm 1})=L_{\mp 1}$, there exist $\mu, \mu'\in\K^{*}$ so that $\phi (x)=\mu y$ and  $\phi (y)=\mu' x$. This time, if we apply $\phi$ to relation $xy-r^{-1}yx+f(h)=0$ and work it out in 
$L$ using~(\ref{E:aut:gwa:gdua2}), we arrive at the equivalent equation
\begin{equation}\label{E:grautL:another:srinv}
\mu\mu'(1-r^{-2})k+f(\lambda k)=\mu\mu'(r^{-1}g(rh)-g(h)).
\end{equation}
So each one of the two sides of~(\ref{E:grautL:another:srinv}) must be a scalar. In particular, the condition $\mu\mu'(1-r^{-2})k+f(\lambda k)\in\K$ implies $\deg (f)=\deg (g)=1$, as $r$ is not a root of $1$. 
Thus $f(X)=\alpha X+\beta$, for $\alpha, \beta\in\K$ with $\alpha\neq 0$ and 
$g(X)=\frac{\alpha}{r^{-1}-r}X+\frac{\beta}{r^{-1}-1}$. Equation~(\ref{E:grautL:another:srinv}) becomes 
$\left( \mu\mu' (1-r^{-2})+\lambda\alpha\right) k+\beta=\mu\mu'\beta$, which is equivalent to $\mu'=\frac{r\alpha\lambda}{\mu (s-r)}$ and $\beta (\frac{r\alpha\lambda}{s-r} -1)=0$. So 
$\phi=\psi^{-}_{(\mu, \lambda )}$, as we wished to conclude.
\end{proof}

We will continue our study of those automorphisms of $L$ which are not in the subgroup $\h$. The following proposition will be useful in the proof of 
Proposition~\ref{P:grautL:htoxn}. Lacking a precise reference for part (a), we provide a simple sketch of the proof. Note, however, that parts (b) and (c) can be deduced from~\cite[Prop.\ 1.4.4]{AC92}.

In what follows, given $q\in\K\setminus\{ 0, 1\}$, $\K_{q}[z, w]$ denotes the quantum plane, generated over $\K$ by indeterminates $z$, $w$ satisfying the $q$-commutation relation $zw=qwz$.

\begin{proposition}\label{P:grautL:qplane}
Let $q, q'\in\K\setminus\{ 0, 1\}$ and assume $\phi : \K_{q}[z, w]\rightarrow\K_{q'}[z', w']$ is an algebra isomorphism. Then:
\begin{enumerate}
\item $q'=q^{\pm 1}$;
\item if $q\neq -1$ then either $q'=q$, $\phi (z)=\lambda_{1}z'$, $\phi (w)=\lambda_{2}w'$ or $q'=q^{-1}$, $\phi (z)=\lambda_{1}w'$, $\phi (w)=\lambda_{2}z'$, for $\lambda_{1}, \lambda_{2}\in\K^{*}$;
\item if $q=-1$ then either  $\phi (z)=\lambda_{1}z'$, $\phi (w)=\lambda_{2}w'$ or  
$\phi (z)=\lambda_{1}w'$, $\phi (w)=\lambda_{2}z'$, for $\lambda_{1}, \lambda_{2}\in\K^{*}$.
\end{enumerate}
\end{proposition}

\begin{proof}
It can easily be shown that the set of normal elements of $\K_{q'}[z', w']$ is 
$\{ z'^{a}w'^{b}v\mid a, b\in\N, v\in\Z\left( \K_{q'}[z', w']\right) \}$. Therefore, 
$\phi (z)=z'^{a}w'^{b}v$ for some $a, b\in\N$ and $v\in\Z\left( \K_{q'}[z', w']\right)$. If $a=0=b$ then $\phi (z)$ would be central, implying that also $z$ is central and $q=1$, which is a contradiction. Since $z$ generates a completely prime ideal of $\K_{q}[z, w]$, it must be that $v=\lambda_{1}\in\K^{*}$ and either $a=1$, $b=0$ or $a=0$, $b=1$. A similar statement holds for $\phi(w)$. If $\phi (z)=\lambda_{1}z'$ then, by the surjectivity of $\phi$, $\phi (w)=\lambda_{2}w'$ for some $\lambda_{2}\in\K^{*}$. By relation $zw=qwz$, it follows that $q'=q$. Similarly, if $\phi (z)=\lambda_{1}w'$ then $q'=q^{-1}$.
\end{proof}

The proof of the next lemma is omitted, as it is obvious.

\begin{lemma}\label{L:grautL:some4f0}
Assume $f=0$ and $s=r^{-1}$. There are  $\phi, \psi\in\aut$, uniquely determined by the rules:
\begin{align*}
\phi (x)=\lambda_{1}y, \q\q \phi (y)=\lambda_{2}h \q \text{and} \q \phi (h)=\lambda_{3}x,\\
\psi (x)=\lambda_{1}h, \q\q \psi (y)=\lambda_{2}x \q \text{and} \q \psi (h)=\lambda_{3}y,
\end{align*}
for any $(\lambda_{1}, \lambda_{2}, \lambda_{3})\in\left(\K^{*}\right )^{3}$.
\end{lemma}

\begin{proposition}\label{P:grautL:htoxn}
Suppose there exists $\phi\in\aut$ satisfying $\phi (h)=\mu x^{n}$ 
(resp.\ $\phi (h)=\mu y^{n}$), for some $n\geq 1$ and $\mu\in\K^{*}$. Then:
\begin{enumerate}
\item $n=1$, $f=0$, $s=r^{-1}$;
\item $\phi (x)=\lambda_{1}y$ and $\phi (y)=\lambda_{2}h$ (resp.\ $\phi (x)=\lambda_{1}h$ and $\phi (y)=\lambda_{2}x$), for some $\lambda_{1}, \lambda_{2}\in\K^{*}$.
\end{enumerate}
\end{proposition}

\begin{proof}
Let $\phi\in\aut$ with $\phi (h)=\mu x^{n}$, $n\geq 1$ and $\mu\in\K^{*}$. By 
Lemmas~\ref{L:propautL:phih} and~\ref{L:propautL:phihnotxn}, $n=1$ and $f=0$. Then $\phi$ induces an isomorphism $\bar{\phi}: L/\langle h\rangle \rightarrow L/\langle x\rangle$. Clearly, $L/\langle h\rangle\simeq\K_{s}[\bar{x}, \bar{y}]$ and 
$L/\langle x\rangle\simeq\K_{r}[\bar{h}, \bar{y}]$, so Proposition~\ref{P:grautL:qplane} implies that $s=r^{\pm 1}$.

Suppose first that $s=r$. Then, again by Proposition~\ref{P:grautL:qplane} ($r\neq -1$), there are $\lambda_{1}, \lambda_{2}\in\K^{*}$ so that 
$\bar{\phi}(\bar{x})=\lambda_{1}\bar{h}$ and $\bar{\phi}(\bar{y})=\lambda_{2}\bar{y}$. Hence there exist also $v_{1}, v_{2}\in L$ such that $\phi (x)=\lambda_{1}h+v_{1}x$ and $\phi (y)=\lambda_{2}y+v_{2}x$. If we apply $\phi$ to the relation $xh=rhx$ and compute, we arrive at the relation $\lambda_{1}(1-r^{2})h=rxv_{1}-v_{1}x$. This implies the contradiction $h\in\langle x\rangle$. Therefore it must be that $s=r^{-1}$.

Proceeding in the same manner as in the last paragraph, we conclude that $\phi (x)=\lambda_{1}y+v_{1}x$ and $\phi (y)=\lambda_{2}h+v_{2}x$. Furthermore, relation $xh=rhx$ implies 
$v_{1}x=rxv_{1}$. Writing $v_{1}$ in the basis $\left\{ x^{\alpha}y^{\beta}h^{\gamma} \right\}_{\alpha, \beta, \gamma\geq 0}$, we easily deduce that $v_{1}$ can be written as 
$v_{1}=y\xi$, for some $\xi\in L$. Thus $\phi (x)=y(\lambda_{1}+\xi x)$. Since $\phi (x)$ must be normal and generate a completely prime ideal of $L$, as $x$ does, it follows that 
$\lambda_{1}+\xi x$ is normal (because both $y(\lambda_{1}+\xi x)$ and $y$ are normal, and $L$ is a domain). We can then invoke Lemma~\ref{L:aut:norm:prodnormpri} and infer that $\lambda_{1}+\xi x\in\K^{*}$. In such a case, necessarily $\xi=0$ and 
$\phi (x)=\lambda_{1}y$. Similarly, $\phi (y)=\lambda_{2}h$.

The case $\phi (h)=\mu y^{n}$ is symmetric.
\end{proof}

\begin{proposition}\label{P:grautL:htoNd0dl}
Suppose $\tau, \epsilon>0$ and $\phi (h)=\sum_{i=0}^{l}d_{i}\left( h^{\epsilon} \right)^{l-i}\left( k^{\tau} \right)^{i}$, for some $\phi\in\aut$ with $l>0$, $d_{i}\in\K$  and $d_{0}, d_{l}\neq 0$. Then $\epsilon=l=1$, $f(X)=\beta\in\K$ and $\phi=\psi^{+}_{(\mu, \mu', d_{0}, d_{1})}$ for some $\mu, \mu'\in\K^{*}$ so that  $\beta(\mu \mu'-1)=0$.
\end{proposition}

\begin{proof}
Let $\phi\in\aut$ and assume the hypotheses of the proposition are satisfied. For simplicity of notation write $N:=\sum_{i=0}^{l}d_{i}\left( h^{\epsilon} \right)^{l-i}\left( k^{\tau} \right)^{i}$. Then $xN=r^{\epsilon l}Nx$ and $yN=r^{-\epsilon l}Ny$. By 
Lemma~\ref{L:propautL:eigenv}, $r$ and $r^{\epsilon l}$ generate the same multiplicative subgroup of $\K^{*}$ and hence, $r$ not being a root of unity, $\epsilon l=\pm 1$. As both integers $\epsilon$ and $l$ are positive, it must be that $\epsilon=l=1$. In particular, $\phi (h)=N=d_{0}h+d_{1}k^{\tau}$, $xN=rNx$ and $yN=r^{-1}Ny$. As we have argued  in the proof of Proposition~\ref{P:grautL:allh}, this implies that $\phi (L_{i})=L_{i}$ for all $i\in\Z$, and that $\phi (x)=\mu x$, 
$\phi (y)=\mu' y$ for some $\mu, \mu'\in\K^{*}$.

If we apply $\phi$ to the relation $xy-syx+f(h)=0$ and simplify, we obtain $\mu\mu'f(h)=f(N)$. Since the left-hand side of the latter equation is a polynomial in $h$ and $N=d_{0}h+d_{1}k^{\tau}$ with $\tau>0$ and $d_{1}\neq 0$, by hypothesis, the given relation forces $f$ to be a constant polynomial, say $f(X)=\beta\in\K$. In that case, equation $\mu\mu'f(h)=f(N)$ reduces to $\beta(\mu\mu'-1)=0$ and $\phi$ must be the automorphism $\phi=\psi^{+}_{(\mu, \mu', d_{0}, d_{1})}$ of 
Lemma~\ref{L:grautL:someother}(a).
\end{proof}

\subsection{The group  $\aut$}\label{SS:thegp}

Using the information obtained this far, especially in Section~\ref{SS:grautL}, the  group 
$\aut$ is   explicitly determined in the following theorem. Recall the definition of $\tau$ and $\epsilon$ given in Section~\ref{SS:cent}.

\begin{theorem}\label{T:thegp:main} 
Let $L=L(f, r, s, \gamma)$ be a generalized down-up algebra. 
Assume $r$, $s\in\K^{*}$, $r$ is not a root of unity and $f(X)=sg(X)-g(rX-\gamma)$ for some $g\in\K[X]$. Then the group $\aut$ of algebra automorphisms of $L$ is isomorphic to:
\begin{enumerate}
\item $\left(\K^{*} \right)^{3}\rtimes \Z/3\Z$ if $f=0$ and $s=r^{-1}$, where the generator 
$1+3\Z$ of $\Z/3\Z$ acts on the torus $\left(\K^{*} \right)^{3}$ via the automorphism $(\lambda_{1}, \lambda_{2}, \lambda_{3})\mapsto (\lambda_{3}, \lambda_{1}, \lambda_{2})$;

\item $\K\rtimes\left(\K^{*} \right)^{3}$ if $f=0$ and $s^{\tau}=r$, where $(\lambda_{1}, \lambda_{2}, \lambda_{3})\in\left(\K^{*} \right)^{3}$ acts on the additive group $\K$ via the automorphism $t\mapsto \lambda_{1}^{-1} \left(\lambda_{2} \lambda_{3}\right)^{\tau}t$;

\item $\K\rtimes\left(\K^{*} \right)^{2}$ if $\deg (f)=0$ and $s^{\tau}=r$, where $(\lambda_{1}, \lambda_{2})\in\left(\K^{*} \right)^{2}$ acts on the additive group $\K$ via the automorphism $t\mapsto \lambda_{1}^{-1} t$;

\item $\left(\K^{*} \right)^{2}\rtimes \Z/2\Z$ if $\deg (f)=1$, $s=r^{-1}$ and $f(\frac{\gamma}{r-1})=0$, where the generator 
$1+2\Z$ of $\Z/2\Z$ acts on the torus $\left(\K^{*} \right)^{2}$ via the automorphism $(\lambda_{1}, \lambda_{2})\mapsto (\lambda_{1}, \lambda_{2}^{-1} \lambda_{1})$;

\item $\K^{*}\rtimes \Z/2\Z$ if $\deg (f)=1$, $s=r^{-1}$ and $f(\frac{\gamma}{r-1})\neq 0$, where the generator 
$1+2\Z$ of $\Z/2\Z$ acts on the torus $\K^{*}$ via the automorphism $\lambda \mapsto \lambda^{-1}$;

\item $\h$ otherwise, where $\h$ should be taken as  described in Proposition~\ref{P:grautL:allh}, with $\rho$ determined with respect to the polynomial $\tilde{f}(X)=f\left( \frac{1}{r-1}(X+\gamma) \right)$.
\end{enumerate}
\end{theorem}

In view of Proposition~\ref{P:conf:gamma0rnot1} and the hypothesis that $r$ is not a root of unity, we can assume that $\gamma=0$, by replacing $f\in\K[X]$ with $\tilde{f}(X)=f\left( \frac{1}{r-1}(X+\gamma) \right)$. Notice that $f=0$ \emph{iff} $\tilde{f}=0$, $\deg (f)=\deg (\tilde{f})$ and $f(\frac{\gamma}{r-1})=0$ \emph{iff} $\tilde{f}(0)=0$. Hence, for the proof of Theorem~\ref{T:thegp:main}, we assume that 
$\gamma=0$ and  $\h$ is the subgroup of $\aut$ defined in Section~\ref{SS:grautL} and computed in  Proposition~\ref{P:grautL:allh}, relative to $\tilde{f}$.  For a group $G$, $\langle\langle a\rangle\rangle$ denotes the cyclic subgroup of $G$ generated by $a\in G$.

For the sake of clarity, we split the proof of this theorem into three propositions, dealing separately  with the cases $f=0$; $f\neq 0$, $\epsilon>0$; and $f\neq 0$, $\epsilon\leq 0$. Each of these propositions gives additional insight into the group $\aut$, as it explicitly lists the elements of this group, rather than just describing the group up to isomorphism.

\begin{proposition}\label{P:thegp:f0} 
Let $L=L(f, r, s, 0)$ be as before and suppose $f=0$. Then:
\begin{enumerate}
\item $\aut=\h\rtimes \langle\langle \phi\rangle\rangle$ if $s=r^{-1}$, where $\phi$ is  defined by  $\phi (x)=y$, $\phi(y)=h$, $\phi(h)=x$, as given in Lemma~\ref{L:grautL:some4f0}; for all $\phi_{(\lambda_{1}, \lambda_{2}, \lambda_{3})}\in\h$, $\phi\circ\phi_{(\lambda_{1}, \lambda_{2}, \lambda_{3})}\circ\phi^{-1}=\phi_{(\lambda_{3}, \lambda_{1}, \lambda_{2})}$.

\item $\aut=\{ \psi^{+}_{(1, 1, 1, t)}\mid t\in\K \}\rtimes \h$ if $s^{\tau}=r$, where $\psi^{+}_{(1, 1, 1, t)}$ is  given in Lemma~\ref{L:grautL:someother}(a); for all $\phi_{(\lambda_{1}, \lambda_{2}, \lambda_{3})}\in\h$, $\phi_{(\lambda_{1}, \lambda_{2}, \lambda_{3})}\circ\psi^{+}_{(1, 1, 1, t)}\circ\phi_{(\lambda_{1}, \lambda_{2}, \lambda_{3})}^{-1}=\psi^{+}_{(1, 1, 1, \lambda_{1}^{-1} \left(\lambda_{2} \lambda_{3}\right)^{\tau}t)}$.

\item $\aut=\h$ otherwise.
\end{enumerate}
\end{proposition}

\begin{proof}
Let $\phi\in\aut$. By Lemma~\ref{L:propautL:phih}, $\phi (h)$ is normal and generates a (completely) prime ideal of $L$. Hence, by Proposition~\ref{P:aut:norm:allnorm}, $\phi (h)=h^{a}k^{b}qx^{n}$ or $\phi (h)=h^{a}k^{b}qy^{n}$, for $a, b, n\geq 0$ and $q\in D$ as described in Lemma~\ref{L:norm:norminD}. Observing Lemma~\ref{L:aut:norm:prodnormpri}, we see that one of the following must occur:
\begin{equation*}
\phi (h)=\lambda h, \q \phi (h)=\lambda k, \q \phi (h)=q, \q \phi (h)=\lambda x \q \text{or}\q \phi (h)=\lambda y,
\end{equation*}
for some $\lambda\in\K^{*}$.

The case $\phi (h)=\lambda k$ cannot occur, by Proposition~\ref{P:grautL:another}, as $f=0$.
Notice that, since $h$ is not central ($r\neq 1$), if $\phi (h)=q$ then necessarily $\tau, \epsilon>0$ and 
$q=\sum_{i=0}^{l}d_{i}\left( h^{\epsilon} \right)^{l-i}\left( k^{\tau} \right)^{i}$, with $l>0$ and $d_{0}, d_{l}\neq 0$. Furthermore, in this case, Proposition~\ref{P:grautL:htoNd0dl} can be applied and we deduce that 
$\epsilon=1=l$ and $\phi=\psi^{+}_{(\mu, \mu', d_{0}, d_{1})}$ for  $\mu, \mu'\in\K^{*}$.
Similarly, if either $\phi (h)=\lambda x$ or $\phi (h)=\lambda y$ then Proposition~\ref{P:grautL:htoxn} implies that $(\tau, \epsilon)=(1, -1)$.

Suppose first that neither $\epsilon=1$ nor $(\tau, \epsilon)=(1, -1)$. Then the only possibility is $\phi (h)=\lambda h$ and $\phi\in\h$, as given in Proposition~\ref{P:grautL:allh}(a).

Now consider the case $\epsilon=1$. Then either $\phi (h)=\lambda h$ and $\phi\in\h$, or 
$\phi=\psi^{+}_{(\mu, \mu', d_{0}, d_{1})}$ for  $\mu, \mu'\in\K^{*}$. As $\psi^{+}_{(\mu, \mu', d_{0}, d_{1})}=\phi_{(d_{0}, \mu, \mu')}\circ \psi^{+}_{(1, 1, 1, t)}$, with $\phi_{(d_{0}, \mu, \mu')}\in\h$ and $t=d_{1}(\mu\mu')^{-\tau}$, we can assume without loss of generality that $\phi=\psi^{+}_{(1, 1, 1, t)}$, for some $t\in\K$. Now note that $\psi^{+}_{(1, 1, 1, 0)}=\mathrm{id}_{L}$ and $\psi^{+}_{(1, 1, 1, t)}\circ\psi^{+}_{(1, 1, 1, t')}=\psi^{+}_{(1, 1, 1, t+t')}$, so $\{ \psi^{+}_{(1, 1, 1, t)}\mid t\in\K \}$ is isomorphic to the additive group of $\K$. Furthermore, $\h\cap\{ \psi^{+}_{(1, 1, 1, t)}\mid t\in\K \}=\{ \mathrm{id}_{L}\}$ and it is routine to verify that 
\begin{equation*}
\phi_{(\lambda_{1}, \lambda_{2}, \lambda_{3})}\circ\psi^{+}_{(1, 1, 1, t)}\circ\phi_{(\lambda_{1}, \lambda_{2}, \lambda_{3})}^{-1}=\psi^{+}_{(1, 1, 1, \lambda_{1}^{-1} \left(\lambda_{2} \lambda_{3}\right)^{\tau}t)}.
\end{equation*}
So indeed $\aut=\{ \psi^{+}_{(1, 1, 1, t)}\mid t\in\K \}\rtimes \h\simeq\K\rtimes\left( \K^{*}\right)^{3}$, in this case.

Finally, let us consider the case $(\tau, \epsilon)=(1, -1)$, i.e., $s=r^{-1}$. As we have seen, either $\phi (h)=\lambda h$ and $\phi\in\h$ or $\phi (h)=\lambda x$ or $\phi (h)=\lambda y$. Assume that $\phi (h)=\lambda x$ (the case $\phi (h)=\lambda y$ is symmetric). Then by Proposition~\ref{P:grautL:htoxn}, $\phi (x)=\lambda_{1}y$ and $\phi (y)=\lambda_{2}h$. Composing $\phi$ with an appropriate element of $\h$, we can assume that $\phi (h)=x$, $\phi (x)=y$ and $\phi (y)=h$. Then $\phi^{2}(h)=y$ and $\phi^{3}=\mathrm{id}_{L}$. Hence 
$\langle\langle \phi\rangle\rangle$ is the cyclic group of order $3$, $\h\cap\langle\langle \phi\rangle\rangle=
\{ \mathrm{id}_{L} \}$ and 
\begin{equation*}
\phi\circ\phi_{(\lambda_{1}, \lambda_{2}, \lambda_{3})}\circ\phi^{-1}=\phi_{(\lambda_{3}, \lambda_{1}, \lambda_{2})}.
\end{equation*}
This proves that, in this case, $\aut=\h\rtimes \langle\langle \phi\rangle\rangle$.
\end{proof}

\begin{proposition}\label{P:thegp:fn0epos} 
Let $L=L(f, r, s, 0)$ be as before and suppose $f\neq0$ and $\epsilon>0$. Then:
\begin{enumerate}
\item $\aut=\{ \psi^{+}_{(1, 1, 1, t)}\mid t\in\K \}\rtimes \h$ if $s^{\tau}=r$ and $\deg (f)=0$, where $\psi^{+}_{(1, 1, 1, t)}$ is  given in Lemma~\ref{L:grautL:someother}(a); for all $\phi_{(\lambda_{1}, \lambda_{2})}\in\h$, $\phi_{(\lambda_{1}, \lambda_{2})}\circ\psi^{+}_{(1, 1, 1, t)}\circ\phi_{(\lambda_{1}, \lambda_{2})}^{-1}=\psi^{+}_{(1, 1, 1, \lambda_{1}^{-1}t)}$.

\item $\aut=\h$ otherwise.
\end{enumerate}
\end{proposition}

\begin{proof}
Let $\phi\in\aut$. As in the proof of Proposition~\ref{P:thegp:f0}, only two possibilities can occur: $\phi (h)=\lambda h$ or $\phi (h)=q$, with $q\in D$ as described in Lemma~\ref{L:norm:norminD}(c) (see Proposition~\ref{P:grautL:another} and Proposition~\ref{P:grautL:htoxn}).

If $\phi (h)=\lambda h$ then $\phi\in\h$. Otherwise, Proposition~\ref{P:grautL:htoNd0dl} implies that $\epsilon=1$, $\deg (f)=0$ and $\phi=\psi^{+}_{(\mu, \mu^{-1}, d_{0}, d_{1})}$ for $\mu, d_{0}, d_{1}\in\K^{*}$. Therefore, if either $\epsilon\neq 1$ or $\deg (f)\neq 0$ then $\aut=\h$. If $\epsilon=1$ and $\deg (f)=0$ then 
$\rho=0$ and $\h=\{ \phi_{(\lambda_{1}, \lambda_{2})}\mid (\lambda_{1}, \lambda_{2})\in\left(\K^{*}\right)^{2} \}$, with 
$\phi_{(\lambda_{1}, \lambda_{2})}(h)=\lambda_{1}h$, 
$\phi_{(\lambda_{1}, \lambda_{2})}(x)=\lambda_{2}x$,  $\phi_{(\lambda_{1}, \lambda_{2})}(y)=\lambda_{2}^{-1}y$. In case $\phi\notin\h$  we can assume $\phi=\psi^{+}_{(1, 1, 1, t)}$ and proceed as in the proof of Proposition~\ref{P:thegp:f0} to conclude that $\aut=\{ \psi^{+}_{(1, 1, 1, t)}\mid t\in\K \}\rtimes \h$.
\end{proof}

\begin{proposition}\label{P:thegp:fn0eneg}
Let $L=L(f, r, s, 0)$ be as before and suppose $f\neq0$ and $\epsilon\leq 0$. Then:
\begin{enumerate}
\item $\aut=\h\rtimes \langle\langle \psi^{-}_{(1, \frac{r^{-1}-r}{r\alpha})}\rangle\rangle$ if $s=r^{-1}$ and $f(X)=\alpha X$ for $\alpha\in\K^{*}$, where $\psi^{-}_{(1, \frac{r^{-1}-r}{r\alpha})}$ is  given in Lemma~\ref{L:grautL:someother}(b); for all $\phi_{(\lambda_{1}, \lambda_{2})}\in\h$, $\psi^{-}_{(1, \frac{r^{-1}-r}{r\alpha})}\circ\phi_{(\lambda_{1}, \lambda_{2})}\circ (\psi^{-}_{(1, \frac{r^{-1}-r}{r\alpha})})^{-1}=\phi_{(\lambda_{1}, \lambda_{2}^{-1}\lambda_{1})}$.

\item $\aut=\h\rtimes \langle\langle \psi^{-}_{(1, \frac{r^{-1}-r}{r\alpha})}\rangle\rangle$ if $s=r^{-1}$ and $f(X)=\alpha X+\beta$ for $\alpha, \beta\in\K^{*}$, where $\psi^{-}_{(1, \frac{r^{-1}-r}{r\alpha})}$ is  given in Lemma~\ref{L:grautL:someother}(b); for all $\phi_{(1, \lambda)}\in\h$, $\psi^{-}_{(1, \frac{r^{-1}-r}{r\alpha})}\circ\phi_{(1, \lambda)}\circ(\psi^{-}_{(1, \frac{r^{-1}-r}{r\alpha})})^{-1}=\phi_{(1, \lambda^{-1})}$.

\item $\aut=\h$ otherwise.
\end{enumerate}
\end{proposition}
\begin{proof}
We only sketch the proof, as it is similar to the proof of the two previous results.

Assume $\phi\notin\h$. Hence, as before, the only other possibility is $\phi (h)=\lambda k$, for some 
$\lambda\in\K^{*}$. Then, by Proposition~\ref{P:grautL:another}, $s=r^{-1}$ and $\deg (f)=1$. Write $f(X)=\alpha X+\beta$, with $\alpha\neq 0$. Thus $\phi=\psi^{-}_{(\mu, \lambda)}$ for $\mu\in\K^{*}$ and $\beta(\frac{r\alpha\lambda}{s-r}-1)=0$.

Suppose $\beta=0$. Then $\rho=0$, $\h=\{ \phi_{(\lambda_{1}, \lambda_{2})}\mid (\lambda_{1}, \lambda_{2})\in\left(\K^{*}\right)^{2} \}$ and it can be assumed that $\phi=\psi^{-}_{(1, \frac{s-r}{r\alpha})}$, as 
$\psi^{-}_{(\mu, \lambda)}=\phi_{(\frac{r\alpha\lambda}{s-r}, \frac{r\alpha\lambda}{\mu (s-r)})}\circ\psi^{-}_{(1, \frac{s-r}{r\alpha})}$. 
The result follows in this case because  
$(\psi^{-}_{(1, \frac{s-r}{r\alpha})})^{2}=\mathrm{id}_{L}$, $\h\cap \langle\langle \psi^{-}_{(1, \frac{s-r}{r\alpha})}\rangle\rangle=\{ \mathrm{id}_{L}\}$ and 
\begin{equation*}
\psi^{-}_{(1, \frac{s-r}{r\alpha})}\circ\phi_{(\lambda_{1}, \lambda_{2})}\circ (\psi^{-}_{(1, \frac{s-r}{r\alpha})})^{-1}=\phi_{(\lambda_{1}, \lambda_{2}^{-1}\lambda_{1})}.
\end{equation*}

The case $\beta\neq 0$ is analogous, with $\rho=1$ and $\frac{r\alpha\lambda}{s-r}=1$.
\end{proof}


\section{Automorphisms of down-up algebras}\label{S:adua}

Having computed in Section~\ref{S:aut} the automorphism group of the generalized down-up algebras $L(f, r, s, \gamma)$ which are conformal, Noetherian and for which $r$ is not a root of unity, we specialize in this section our results to the case of down-up algebras. We remark that the isomorphism problem for Noetherian down-up algebras has already been solved in~\cite{CM00}.

Other classes of algebras to which our study applies are Le Bruyn's \emph{conformal $\mathfrak{sl}_{2}$ enveloping algebras}~\cite{lLB95}, occuring as $L(bx^{2}+x, r, s, \gamma)$, for $b\in\K$ and $rs\neq 0$, and some of Witten's seven parameter deformations of the enveloping algebra of $\mathfrak{sl}_{2}$~\cite{eW90} (see also~\cite[Thm.~2.6]{gB99} and~\cite[Ex.~1.4]{CS04}). We leave it to the reader to apply Theorem~\ref{T:thegp:main} to these and perhaps to other classes of generalized down-up algebras.

\subsection{Some well-known examples}\label{SS:adua:swke}

We start by computing some examples, which have appeared elsewhere in the literature, namely \cite{AC92}, \cite{AD96} and \cite{pC95}.

The quantum Heisenberg algebra $\mathbb{H}_{q}$ is a deformation of the enveloping algebra of the $3$-dimensional Heisenberg Lie algebra. It can be viewed as the positive part in the triangular decomposition of the quantized enveloping algebra corresponding to the simple complex Lie algebra $\mathfrak{sl}_{3}$ of traceless $3\times 3$ matrices. It is presented as the unital associative $\K$-algebra generated by $X$, $Y$, $Z$, with relations:
\begin{equation}\label{E:adua:swke:qH}
XZ=qZX, \q\q ZY=qYZ, \q\q XY-q^{-1}YX=Z,
\end{equation}
where $q\in\K^{*}$.

The automorphism group of the quantum Heisenberg algebra was computed  by Caldero~\cite{pC95} in case $q$ is transcendental over $\Q$ and, independently, by Alev and Dumas~\cite{AD96} just assuming $q$ is not a root of $1$. It is the semidirect product of the $2$-torus 
$\left(\K^{*} \right)^{2}$, acting diagonally on the generators $X$ and $Y$, and the finite group $\Z/2\Z$, acting as the symmetric group on $X$ and $Y$. From relations~(\ref{E:adua:swke:qH}), we see that 
$\mathbb{H}_{q}$ is the algebra $L(-X, q, q^{-1}, 0)$, isomorphic to the down-up algebra $A(q+q^{-1}, -1, 0)$. Thus, $L(-X, q, q^{-1}, 0)$ is Noetherian for all choices of $q\in\K^{*}$, and conformal provided $q\neq 1, -1$. If we assume, as in~\cite{AD96}, that $q$ is not a root of $1$, then $\tau=1$, $\epsilon=-1$ and $\rho=0$. Hence, we retrieve~\cite[Prop. 2.3]{AD96} in Theorem~\ref{T:thegp:main}(d).

Another example, which is not that of a down-up algebra, is the algebra of regular functions on quantum affine $3$-space. This is the unital associative $\K$-algebra with generators  $x_{1}$, $x_{2}$, $x_{3}$, satisfying the relations: 
\begin{equation}\label{E:adua:swke:q3sp}
x_{1}x_{2}=q_{12}x_{2}x_{1}, \q\q x_{1}x_{3}=q_{13}x_{3}x_{1}, \q\q x_{2}x_{3}=q_{23}x_{3}x_{2},
\end{equation}
where $q_{12}, q_{13}, q_{23}\in\K^{*}$.
In case $q_{12}q_{13}=1$, this algebra coincides with the generalized down-up algebra $L(0, r, s, 0)$, with $r=q_{13}=q_{12}^{-1}$ and $s=q_{23}$. Therefore, Proposition~\ref{P:thegp:f0} can be used to compute $\aut$ whenever $r$ is not a root of $1$. In particular, if $r=s$ is not a root of unity, it was seen in~\cite[Th\'e. 1.4.6i)]{AC92} that $\aut$ is isomorphic to the semidirect product $\K\rtimes\left(\K^{*} \right)^{3}$. Since, in this case, $\tau=\epsilon=1$, we also obtain this description of $\aut$ in 
Theorem~\ref{T:thegp:main}(b).

\subsection{Down-up algebras}\label{SS:adua:dua}

Using Proposition~\ref{P:conf:gamma0rnot1}, Theorem~\ref{T:thegp:main} and also some results of Carvalho and Musson~\cite{CM00} and  Jordan~\cite{dJ00}, we will now compute the automorphism group of all down-up algebras $A(\alpha, \beta, \gamma)$, except in the cases where either $\beta=0$ or where both roots $r$ and $s$ of the polynomial $X^{2}-\alpha X-\beta$ are roots of $1$. 

Our exceptions include two remarkable examples, corresponding to $r=s=1$. One is $U(\mathfrak{sl}_{2})$, the enveloping algebra of the complex simple Lie algebra $\mathfrak{sl}_{2}$, which occurs in the family of down-up algebras as $A(2, -1, 1)$. The other one is $U(\mathfrak{h})$, the enveloping algebra of the $3$-dimensional Heisenberg Lie algebra, occurring as $A(2, -1, 0)$. Neither for $U(\mathfrak{sl}_{2})$ nor for $U(\mathfrak{h})$ is the full group of automorphisms known, and in both cases \emph{wild} automorphisms have been shown to exist, by work of Joseph~\cite{aJ76} and Alev~\cite{jA86}, respectively. Regarding $U(\mathfrak{sl}_{2})$, Dixmier computed the automorphism group of the minimal primitive quotients of 
this algebra in~\cite{jD73}. As for the primitive quotients of $U(\mathfrak{h})$ which are not one-dimensional, these are isomorphic to the first Weyl algebra ${A}_{1}(\K)$, whose group of automorphisms was also computed by Dixmier in~\cite{jD68}. 

In~\cite{BJ01}, Bavula and Jordan solved the isomorphism problem and found generators for the automorphism group of  generalized Weyl algebras of the form 
$\K[X](X\stackrel{\sigma}{\mapsto}X-1, a)$, a class which includes the infinite-dimensional primitive quotients of both $U(\mathfrak{sl}_{2})$ and $U(\mathfrak{h})$. They also solved the isomorphism problem for Smith's algebras $L(f, 1, 1, 1)$ similar to $U(\mathfrak{sl}_{2})$. We note that our results do not overlap with those of~\cite{BJ01}. Other generalized Weyl algebras of the form 
$\K[X](X\stackrel{\sigma}{\mapsto}qX, a)$, with $q$ not a root of unity, were studied in~\cite{RS06}, and their automorphism group was determined. With minor changes, \cite[Cor.\ 2.2.7]{RS06} can be adapted to describe the automorphism group of the down-up algebras of the form $A(r+1, -r, 0)$, with $r\in\K^{*}$ not a root of unity. This may be achieved by replacing in~\cite[Cor.\ 2.2.7]{RS06} the base field $\K$ by the domain $\K[X]$, and observing that the arguments used are still valid. As a result, we would retrieve a subcase of Theorem~\ref{T:adua:dua:main2}(b) below.

Fix $\alpha$, $\beta$, $\gamma\in\K$ with $\beta\neq 0$, and let $r$ and $s$ be the roots of the polynomial $X^{2}-\alpha X-\beta$ in $\K$. Thus $\alpha=r+s$ and $\beta=-rs$. The down-up algebra $A=A(\alpha, \beta, \gamma)$, as defined in~\cite{BR98}, coincides with $L(X, r, s, \gamma)$, upon 
identifying the canonical generators $d$ and $u$ of $A$ with the generalized Weyl algebra generators 
$x$ and $y$ of $L$, respectively. Since $r$ and $s$ have  symmetric roles in $A$, it should be no surprise that $L(X, r, s, \gamma)\simeq L(X, s, r, \gamma)$, under an isomorphism taking $x$ to $x$, $y$ to $y$ and $h\in L(X, r, s, \gamma)$ to $h+(s-r)yx\in L(X, s, r, \gamma)$. Hence, when dealing with down-up algebras, we can interchange the roles of $r$ and $s$ in $L$. Also, the generator $h$ of $L$ is redundant when $\deg (f)=1$, so in this case it will suffice to give the action of an automorphism of $L$ on the generators $x$ and $y$.

Our results in this section will apply to all down-up algebras $A$ under the restrictions that $rs\neq 0$ and that one of $r$ or $s$ is not a root of $1$.  In view of the symmetric roles of $r$ and $s$, we always assume that $r$ is not a root of $1$. In particular, $A=L(X, r, s, \gamma)\simeq L(\frac{1}{r-1}(X+\gamma), r, s, 0)$, by Proposition~\ref{P:conf:gamma0rnot1}. We distinguish three cases:

\medskip

\emph{Case 1}: $r\neq s$, $r$ is not a root of $1$ and $s\neq 1$. In this case, 
$A\simeq L(\frac{1}{r-1}(X+\gamma), r, s, 0)$ is conformal and $\aaut$ is given in 
Theorem~\ref{T:thegp:main}(d)--(f). If $\gamma=0$ then $\rho=0$; if $\gamma\neq0$ then $\rho=1$.

Assume first that $\gamma=0$. Then $\h\simeq\left(\K^{*} \right)^{2}$ acts diagonally on the generators $x$ and $y$. If, in addition, $s=r^{-1}$ then there is an automorphism of $A$ of order $2$ which interchanges $x$ and $y$.

Now assume $\gamma\neq 0$. Then $\h\simeq\K^{*}$ and $\lambda\in\K^{*}$ acts on $x$ by multiplication by $\lambda$ and on $y$ by multiplication by $\lambda^{-1}$. If, in addition, $s=r^{-1}$ then there is an automorphism of $A$ of order $2$ which interchanges $x$ and $y$.

\smallskip

\emph{Case 2}: $r$ is not a root of $1$ and $s=1$. Assume $\gamma=0$. Then $A=L(X, r, s, 0)$ is conformal, as $r\neq 1$. Also, $\tau=1$, $\epsilon=0$ and $\rho=0$. Thus, by Theorem~\ref{T:thegp:main}, $\aaut=\h\simeq\left(\K^{*} \right)^{2}$, acting diagonally on  $x$ and $y$.

Now assume $\gamma\neq 0$. Then $A=L(X, r, s, \gamma)$ is not conformal, by Proposition~\ref{P:conf:gamma0rnot1} and Lemma~\ref{L:conf:gamma0}, so Theorem~\ref{T:thegp:main} cannot be applied. Let $\omega=yx-xy+\frac{\gamma}{1-r}$. By~\cite[Cor.\ 4.10]{CM00}, any automorphism $\phi$ of $A$ must fix the ideal $\omega A$. By Lemma~\ref{L:noeth:units}, there exists $\lambda\in\K^{*}$ so that $\phi (\omega)=\lambda\omega$. The proof of Proposition~\ref{P:grautL:allh}
can  be readily adapted to show that $\phi (x)=\mu x$ and $\phi (y)=\mu^{-1} y$, for some $\mu\in\K^{*}$. Thus, $\aaut\simeq\K^{*}$.

\smallskip

\emph{Case 3}: $r=s$ is not a root of $1$. In this case, $A=L(X, r, r, \gamma)$ is not conformal, by Proposition~\ref{P:conf:gamma0rnot1} and Lemma~\ref{L:conf:gamma0}. We can use the description of the height one prime ideals of $A$ that appears in~\cite[Prop. 6.13]{dJ00} precisely for the case that $r$ is not a root of one. Indeed, let $\omega=ryx-xy+\frac{\gamma}{1-r}$. Then $\omega\in D$, $x\omega=r\omega x$ and 
$\omega y=ry\omega$, so $\omega$ is normal. If $\gamma=0$ then $\omega A$ is the unique height one prime ideal of $A$. If $\gamma\neq0$ then the height one primes of $A$ are 
$\omega A$ and the annihilators of certain simple finite-dimensional $A$-modules. 
In either case, $\omega A$ is the unique height one prime ideal of $A$ not having finite codimension, as 
$A/\omega A$ is either a quantum plane ($\gamma=0$) or a quantum Weyl algebra ($\gamma\neq0$). 
Thus, all automorphisms of $A$ fix the ideal generated by $\omega$. As above, we deduce that, given $\phi\in\aaut$, there exist nonzero scalars $\lambda$, $\mu$ so that $\phi (x)=\lambda x$ and $\phi (y)=\mu y$. 
If $\gamma=0$, no further restrictions arise on the parameters $\lambda$, $\mu$ and 
$\aaut\simeq\left(\K^{*} \right)^{2}$. In case $\gamma\neq 0$, there is only the additional restriction that $\lambda=\mu^{-1}$, so $\aaut\simeq\K^{*}$.

\medskip

We summarize out results on down-up algebras in the following theorem. For the convenience of those readers who are mostly interested in down-up algebras, we replace our usual generators $x$ and $y$ of $L$ with the canonical generators $d$ and $u$ of $A$, respectively.

\begin{theorem}\label{T:adua:dua:main2}
Let $A=A(\alpha, \beta, \gamma)$ be a down-up algebra, with $\alpha=r+s$ and $\beta=-rs$. Assume that $\beta\neq 0$ and 
that one of $r$ or $s$ is not a root of unity. The group $\aaut$ of algebra automorphisms of $A$ is described bellow.
\begin{enumerate}
\item If $\gamma=0$ and $\beta=-1$ then $\aaut\simeq\left(\K^{*} \right)^{2}\rtimes\Z/2\Z$;
\item If $\gamma=0$ and $\beta\neq-1$ then $\aaut\simeq\left(\K^{*} \right)^{2}$;
\item If $\gamma\neq0$ and $\beta=-1$ then $\aaut\simeq\K^{*}\rtimes\Z/2\Z$;
\item If $\gamma\neq0$ and $\beta\neq-1$ then $\aaut\simeq\K^{*}$.
\end{enumerate}
In all cases, the $2$-torus $\left(\K^{*} \right)^{2}$ acts diagonally on the generators $d$ and $u$, $\mu\in\K^{*}$ acts as multiplication by $\mu$ on $d$ and as multiplication by $\mu^{-1}$ on $u$, and the generator of the finite group $\Z/2\Z$ interchanges $d$ and $u$.
\end{theorem}


\noindent
Paula A.A.B. Carvalho\\
Centro de Matem\'atica da Universidade do Porto,\\
Rua do Campo Alegre 687, 4169-007 Porto, Portugal\\
E-mail: pbcarval@fc.up.pt
 \\[10pt]
Samuel A. Lopes\\
Centro de Matem\'atica da Universidade do Porto,\\
Rua do Campo Alegre 687, 4169-007 Porto, Portugal\\
E-mail: slopes@fc.up.pt

\end{document}